\documentclass[11pt]{amsart}
\usepackage[utf8]{inputenc}
\usepackage{amssymb,amsmath}
\usepackage{graphicx}
\usepackage{color}
\usepackage{tikz}
\usepackage{pgfplots}
\usepackage{enumitem}
\usepackage{diagbox}
\usepackage{mathtools}
\usepackage[hidelinks]{hyperref}
\usepackage{xfrac}

\usepackage{accents}
\usepackage{multicol} 
\usepackage{soul}
\usepackage{subcaption}
\newtheorem{theorem}{Theorem}
\newtheorem{proposition}[theorem]{Proposition}

\theoremstyle{definition}

\newtheorem{lemma}[theorem]{Lemma}

\theoremstyle{remark}
\newtheorem{remark}[theorem]{Remark}

\newtheorem{assumption}[theorem]{Assumption}
\numberwithin{theorem}{section}
\numberwithin{equation}{section}
\numberwithin{table}{section}
\numberwithin{figure}{section}
\newcommand{\V}{\ensuremath{\mathcal{V}}}
\newcommand{\Q}{\ensuremath{\mathcal{Q}}}
\newcommand{\M}{\ensuremath{\mathcal{M}}}

\def\N{\mathbb{N}}

\def\R{\mathbb{R}}

\definecolor{myBlue1}{RGB}{101,149,239}  
\definecolor{myBlue2}{RGB}{113,104,238} 
\definecolor{myBlue3}{RGB}{30,144,255} 
\definecolor{myGreen1}{RGB}{154,204,50} 
\definecolor{myGreen2}{RGB}{69,169,0} 
\definecolor{myGreen3}{RGB}{154,205,50} 
\definecolor{myGreen4}{RGB}{105,139,34} 
\definecolor{myRed1}{RGB}{210,105,30} 
\definecolor{myRed2}{RGB}{165,42,42} 
\definecolor{myRed3}{RGB}{139,26,26} 
\definecolor{lightgray}{RGB}{175,175,175} 
\definecolor{myLGray}{RGB}{225,225,225} 
\definecolor{mycolor0}{rgb}{0.66,0.66,0.66}
\definecolor{mycolor4}{rgb}{0.00000,0.44700,0.74100}
\definecolor{mycolor1}{rgb}{0.85000,0.32500,0.09800}
\definecolor{mycolor2}{rgb}{0.92900,0.69400,0.12500}
\definecolor{mycolor3}{rgb}{0.67000,0.74700,0.14100}
\definecolor{mycolor5}{rgb}{0.49400,0.18400,0.55600}
\definecolor{mycolor7}{rgb}{0.98000,0.82500,0.1800}
\DeclareMathOperator{\trace}{trace}

\DeclareMathOperator{\DBDF}{D_\tau^{BDF}}
\DeclareMathOperator{\Dalt}{D_\tau^{alt}}
\DeclareMathOperator{\EBDF}{E}

\newcommand{\Vh}{\mathbb{V}_h}
\newcommand{\Qh}{\mathbb{Q}_h}
\newcommand{\Mh}{\mathbb{M}_h}
\newcommand{\calKu}{\calK_u}
\newcommand{\calKp}{\calK_p}
\newcommand{\calBu}{\calB_u}
\newcommand{\calBp}{\calB_p}
\newcommand{\Mu}[1][]{M_{u#1}}
\newcommand{\Mp}[1][]{M_{p#1}}
\newcommand{\Ku}[1][]{K_{u#1}}
\newcommand{\Kp}[1][]{K_{p#1}}
\newcommand{\Mla}{M_\lambda}
\newcommand{\Bu}{B_u}
\newcommand{\Bp}{B_p}
\newcommand{\fu}[1][]{f_{#1}}
\newcommand{\fp}[1][]{g_{#1}}
\newcommand{\au}{\alpha}
\newcommand{\ap}{\kappa}
\renewcommand{\c}{\text{c}}

\newcommand{\calB}{\ensuremath{\mathcal{B}} }

\newcommand{\calK}{\ensuremath{\mathcal{K}} }

\newcommand{\calM}{\ensuremath{\mathcal{M}} }

\newcommand{\calP}{\ensuremath{\mathcal{P}} }
\newcommand{\calQ}{\ensuremath{\mathcal{Q}} }

\newcommand{\calT}{\ensuremath{\mathcal{T}} }

\newcommand{\calV}{\ensuremath{\mathcal{V}} }

\def\ds{\,\text{d}s}

\def\dx{\,\text{d}x}
\def\dxi{\,\text{d}\xi}

\newcommand{\partialan}{\ensuremath{\partial_{\au,\nu}} }

\newcommand{\dofOm}{{\ensuremath{N_u}}}

\newcommand{\dofGa}{{\ensuremath{N_p}}}
\newcommand{\dofLam}{{\ensuremath{N_\lambda}}}

\definecolor{corrRed}{RGB}{18,124,175} 

\allowdisplaybreaks
\begin{document}
\title[Second-order bulk--surface splitting]{A second-order bulk--surface splitting for parabolic problems with dynamic boundary conditions}
\author[]{R.~Altmann$^\dagger$, C.~Zimmer$^{\ddagger}$}
\address{${}^{\dagger}$ Institute of Analysis and Numerics, Otto von Guericke University Magdeburg, Universit\"atsplatz 2, 39106 Magdeburg, Germany}
\address{${}^{\ddagger}$ Institute of Mathematics, University of Augsburg, Universit\"atsstr.~12a, 86159 Augsburg, Germany}
\email{robert.altmann@ovgu.de, christoph.zimmer@math.uni-augsburg.de}
\thanks{This paper is accepted for publication in IMA J. Numer. Anal. \\
\indent 
The authors acknowledge the support of the Deutsche Forschungsgemeinschaft (DFG, German Research Foundation) through the project 446856041. Moreover, major parts of this work were carried out while the first author was affiliated with the Institute of Mathematics and the Centre for Advanced Analytics and Predictive Sciences (CAAPS) at the University of Augsburg. }
\date{\today}
\keywords{}
\begin{abstract}
This paper introduces a novel approach for the construction of bulk--surface splitting schemes for semi-linear parabolic partial differential equations with dynamic boundary conditions. The proposed construction is based on a reformulation of the system as a partial differential--algebraic equation and the inclusion of certain delay terms for the decoupling. To obtain a fully discrete scheme, the splitting approach is combined with finite elements in space and a BDF discretization in time. Within this paper, we focus on the second-order case, resulting in a~$3$-step scheme. We prove second-order convergence under the assumption of a weak CFL-type condition and confirm the theoretical findings by numerical experiments. Moreover, we illustrate the potential for higher-order splitting schemes numerically. 
\end{abstract}
%
%
\maketitle
%
{\tiny {\bf Key words.} dynamic boundary conditions, bulk--surface splitting, parabolic PDE, PDAE, delay}\\
\indent
{\tiny {\bf AMS subject classifications.}  {\bf 65M12}, {\bf 65L80}, {\bf 65M20}. } 
%
%
%
\section{Introduction}
In recent years, so-called {\em dynamic boundary conditions} have attracted rising attention as they enable the reflection of effective properties on the surface of the domain~\cite{Esc93,Gol06}. 
For this, the energy balance and the constitutive laws on the boundary need to be taken into account rather than being neglected as for standard boundary conditions. In the here considered parabolic setting, they particularly provide a proper way to model a heat source or a heat transfer on the boundary~\cite{Gol06}. Further applications are short range wall--mixture interactions in the phase separation of multi-component alloy systems as well as the modeling of boundary layers when the thickness goes to zero~\cite{Lie13}. 

From a mathematical point of view, dynamic boundary conditions define a differential equation on the boundary such that the overall system may be interpreted as a coupled system; see~\cite[Ch.~5.3]{Las02} or~\cite{EngF05}. We follow this paradigm and consider the formulation introduced in~\cite{Alt19} as {\em partial differential--algebraic equation} (PDAE). Since such a formulation allows different discretizations in the bulk and on the surface, it permits significant computational savings in cases where the solution on the boundary oscillates rapidly, e.g., due to nonlinearities or heterogeneities~\cite{AltV21}. 
%
In this paper, we are interested in bulk--surface splitting schemes, which decouple and hence further facilitate the numerical approximation. Quite recently, splitting approaches of first order were introduced in~\cite{AltKZ22,CsoFK23}. To the best of our knowledge, these two papers feature the first bulk--surface splitting schemes with a rigorous convergence analysis, using variational and semigroup techniques, respectively. On the other hand, two different Strang splitting schemes were proposed in~\cite{KovL17}. These schemes, however, show sub-optimal convergence rates and may even converge to a perturbed (and hence wrong) solution; see the numerical experiments in~\cite{AltKZ22}. 

The aim of this paper is to remedy the lack of a second-order bulk--surface splitting scheme in the parabolic setting. For the here presented construction of higher-order schemes, we follow the seminal idea of including delay terms as proposed in~\cite{AltMU21,AltMU22ppt} in the context of coupled elliptic--parabolic problems. Here, delay terms are used to approximate the constraints, which couple bulk and surface dynamics. Depending on the desired order, a certain number of delay terms is necessary and the corresponding delay times are multiples of the intended step size of the splitting. Hence, the proposed splitting approach results in multistep methods. Within this paper, we focus on a $3$-step splitting scheme, which is of order two. To obtain a fully discrete scheme, the two subflows are discretized by finite elements in space and BDF-$2$ in time. 
%
To prove second-order convergence, we compare the fully discrete solution with the semi-discrete solution resulting from the spatial discretization alone. We first analyze an auxiliary scheme and rewrite it in form of a perturbed BDF-$2$ discretization of the original system. In order to guarantee convergence, we require a weak CFL-type condition of the form $\tau\le C\, h$ with $\tau$ denoting the temporal step size and~$h$ the spatial mesh width. Note that this is not the typical CFL condition for parabolic problems, where~$\tau$ is restricted by~$h^2$. Based on the convergence result for the auxiliary scheme, we finally conclude second-order convergence for the proposed bulk--surface splitting. Moreover, numerical experiments indicate that no CFL condition is necessary in practice. 

The paper is outlined as follows. In Section~\ref{sect:model} we introduce the model problem, its formulation as PDAE as well as its spatial discretization. The splitting into a bulk and a surface problem is then discussed in Section~\ref{sect:splitting}. In order to allow a decoupling, we modify the coupling conditions by certain delay terms, leading to a bulk--surface splitting scheme. The proof of convergence is then subject of Section~\ref{sect:convergence}. Finally, we complete the paper with a numerical validation of the theoretical results including an outlook to higher-order splitting schemes in Section~\ref{sect:numerics} and some concluding remarks in Section~\ref{sect:conclusion}. 
%
%
\section{Abstract Formulation and Spatial Discretization}\label{sect:model} 
Given a bounded Lipschitz domain $\Omega \subseteq \R^d$, $d\in\{2,3\}$, we consider the following semi-linear parabolic model problem with dynamic boundary conditions: seek $u\colon [0,T] \times \Omega \to \R$ such that 
\begin{subequations}
\label{eq:parabolicDynBC}
\begin{align}
	\dot u - \nabla\cdot(\au \nabla u) 
	&= \fu(u) \qquad\text{in } \Omega \label{eq:parabolicDynBC:eqn} \\
	\dot u - \nabla_\Gamma\cdot(\ap \nabla_\Gamma u) + \partialan u 
	&= \fp(u) \qquad\text{on } \Gamma \coloneqq \partial\Omega. \label{eq:parabolicDynBC:bc}
\end{align}
\end{subequations}
with initial condition~$u(0) = u^0$ and normal derivative $\partialan u = \nu\cdot(\au \nabla u)$, where~$\nu$ denotes the unit normal vector. Moreover, $\nabla_\Gamma$ denotes the gradient in tangential direction of the boundary~$\Gamma$; cf.~\cite[Ch.~16.1]{GilT01}. Throughout this paper, we assume the diffusion coefficients~$\au\in L^\infty(\Omega)$ and~$\ap\in L^\infty(\Gamma)$ to be uniformly positive and the right-hand sides~$\fu$, $\fp$ to be sufficiently smooth. 

%
%
\subsection{PDAE formulation}
For the numerical approximation, we pass to a weak formulation of~\eqref{eq:parabolicDynBC}. One possible way is presented in~\cite{KovL17}, showing that these problems can be understood in the functional analytic framework of standard parabolic systems. For the design of splitting methods, however, it has been observed in~\cite{AltKZ22} that the formulation as PDAE as proposed in~\cite{Alt19} is more eligible. For this, we introduce the function spaces 
\[
	\V \coloneqq H^1(\Omega), \qquad
	\Q \coloneqq H^1(\Gamma), \qquad 
	\M \coloneqq H^{-\sfrac 12}(\Gamma).
\]
Moreover, we extend the system by an auxiliary variable $p\coloneqq \trace u$ (with the standard trace operator) and a Lagrange multiplier~$\lambda$ enforcing this connection of~$u$ and~$p$ on the boundary. In total, this leads to the following weak formulation: find $u\colon [0,T] \to \V$, $p\colon [0,T] \to \Q$, and~$\lambda\colon [0,T] \to \M$ such that  
\begin{subequations}
\label{eq:PDAE}
\begin{align}
	\begin{bmatrix} \dot u \\ \dot p  \end{bmatrix}
	+ \begin{bmatrix} \calKu &  \\  & \calKp \end{bmatrix}
	\begin{bmatrix} u \\ p \end{bmatrix}
	+ \begin{bmatrix} -\calBu^* \\ \calBp^* \end{bmatrix} \lambda 
	&= \begin{bmatrix} \fu(u) \\ \fp(p) \end{bmatrix} \qquad \text{in } (\V\times \Q)^*, \label{eq:PDAE:a} \\
	\calBu u - \calBp p 
	&= \phantom{[} 0\hspace{4.5em} \text{in } \M^*. \label{eq:PDAE:b}
	\end{align}
\end{subequations}
Therein, the differential operators~$\calKu\colon \V \to \V^*$ and~$\calKp\colon \Q \to \Q^*$ are given by
\[
	\langle \calKu u, v\rangle 
	\coloneqq \int_\Omega \au\, \nabla u \cdot \nabla v \dx, \qquad 
	\langle \calKp p, q\rangle
	\coloneqq \int_\Gamma \ap\, \nabla_\Gamma p \cdot \nabla_\Gamma q \dx, 
\]
whereas~$\calBu\colon \V\to\M^*$ denotes the trace operator~$\calBu u \coloneqq \trace u$ and~$\calBp\colon \Q\to\M^*$ the canonical embedding from $H^1(\Gamma)$ to $H^{\sfrac 12}(\Gamma)$. 

Due to the extension of the system, the PDAE~\eqref{eq:PDAE} calls for initial data for~$u$ and~$p$. Hence, we assume given~$u(0) = u^0$ and~$p(0) = p^0$ and refer to the data as {\em consistent} if they coincide on the boundary, i.e., if $\calBu u^0 = \calBp p^0$. 
%
%
\subsection{Spatial discretization} 
First error estimates for finite elements applied to parabolic problems with dynamic boundary conditions such as~\eqref{eq:parabolicDynBC} were already introduced in~\cite{Fai79}. Corresponding bulk--surface finite elements were then analyzed in~\cite{KovL17}. In the present PDAE stetting, the spatial discretization calls for three discrete spaces, namely~$\Vh \subseteq \V$, $\Qh \subseteq \Q$, and~$\Mh \subseteq \M$. Note that we restrict ourselves to conforming finite elements, where the discrete spaces are indeed subspaces. Due to the saddle point structure of the PDAE~\eqref{eq:PDAE}, these spaces need to be compatible in the sense of an {\em inf--sup condition}, which reads  
\begin{equation}
\label{eq:disc_inf_sup}
	\adjustlimits\inf_{\lambda_h\in \Mh\setminus\{0\}} \sup_{(u_h,p_h)\in \Vh\times \Qh\setminus\{0\}} \frac{\langle \calBu u_h - \calBp p_h,\lambda_h\rangle_{\calM^\ast,\calM}}{\|\lambda_h\|_{\calM}\big(\|u_h\|^2_{\calV}+\|p_h\|^2_{\calQ}\big)^{\sfrac 12}} 
	\geq \beta 
	> 0.
\end{equation}
Stable schemes were analyzed in~\cite{AltV21,AltZ23ppt} and include, e.g., triples of the form 
\begin{align}
\label{eq:stableFEM}
	\Vh \coloneqq \calP_k(\calT_{\Omega}) \subseteq \calV, \qquad 
	\Qh \subseteq \calQ, \qquad 
	\Mh\coloneqq \calP_k(\calT_{\Omega}|_{\Gamma}) \subseteq \calM;
\end{align}
see \cite[Lem.~2.5]{AltZ23ppt}. Here, $\calT_{\Omega}$ denotes a quasi-uniform triangulation of the spatial domain~$\Omega$ with maximal mesh width~$h$. This triangulation (and its restriction to the boundary, which is denoted by~$\calT_{\Omega}|_{\Gamma}$) defines the discrete spaces for $u$ and~$\lambda$. For $p$, on the other hand, we may choose any conforming finite element space without loosing stability of the discretization. For the choice of $\Qh$, we can hence focus on approximation properties and introduce a second mesh on the boundary denoted by~$\calT_{\Gamma}$. Based on the applications in mind, we assume that~$\calT_{\Gamma}$ is a refinement of~$\calT_{\Omega}|_{\Gamma}$, which then results in~$\dofLam \coloneqq |\Mh| \le |\Qh| \eqqcolon \dofGa$
\begin{remark}
If we use the same meshes on the boundary, i.e., $\calT_{\Gamma} = \calT_{\Omega}|_\Gamma$, and the same polynomial degree, then we have~$\Qh = \Mh = \Vh|_\Gamma$ and~$\dofGa=\dofLam$. This special case is inf-sup stable and corresponds to the discretization proposed in~\cite{KovL17}.
\end{remark}

The application of a stable finite element scheme to~\eqref{eq:PDAE} results in a differential--algebraic equation (DAE) with unknowns~$u\colon[0,T] \to \R^\dofOm$, $p\colon[0,T] \to \R^\dofGa$, and~$\lambda\colon[0,T] \to \R^\dofLam$. For the sake of readability, we use the same notion for the semi-discrete variables as for the exact solution. The resulting system has the form  
\begin{subequations}
\label{eq:semidiscreteDAE}
\begin{align}
	\begin{bmatrix} \Mu &  \\  & \Mp \end{bmatrix}
	\begin{bmatrix} \dot u \\ \dot p  \end{bmatrix}
	+ \begin{bmatrix} \Ku &  \\  & \Kp \end{bmatrix}
	\begin{bmatrix} u \\ p  \end{bmatrix}
	+ \begin{bmatrix} -\Bu^T \\ \Bp^T \end{bmatrix} \lambda 
	&= \begin{bmatrix} \fu(u) \\ \fp(p) \end{bmatrix}, \label{eq:semidiscreteDAE:a} \\[1mm]
	\Bu u - \Bp p
	&= 0 \label{eq:semidiscreteDAE:b}
\end{align}
\end{subequations}
with symmetric and positive definite mass matrices $\Mu\in \R^{\dofOm\times\dofOm}$, $\Mp\in \R^{\dofGa\times\dofGa}$ as well as symmetric and positive semi-definite stiffness matrices~$\Ku\in \R^{\dofOm\times\dofOm}$, $\Kp\in \R^{\dofGa\times\dofGa}$. Moreover, we have rectangular matrices~$\Bu\in \R^{\dofLam\times\dofOm}$ and $\Bp\in \R^{\dofLam\times\dofGa}$. By the inf-sup condition~\eqref{eq:disc_inf_sup}, matrix~$B_u$ is of full row-rank, see \cite[Th.~2.3]{AltZ23ppt} and \cite[Prop.~II.3.1]{BreF91}. This property implies the invertibility of the matrix~$\Bu^{\vphantom{T}} \Mu^{-1}\Bu^T + \Bp^{\vphantom{T}} \Mp^{-1}\Bp^T$ and, hence, shows that the DAE~\eqref{eq:semidiscreteDAE} has (differentiation) index 2, cf.~\cite[Ch.~VII.1]{HaiW96}. 
%
%
\section{Bulk--surface Splitting Schemes}\label{sect:splitting} 
For the introduction of competitive bulk--surface splitting schemes, we first need to segregate bulk and surface dynamics. This means that we need to identify appropriate subflows acting in the bulk and on the surface, separately. Afterwards, we include delay terms such that a temporal discretization yields a decoupled scheme. 

Throughout this paper, we consider an equidistant partition of the time interval~$[0,T]$ with constant step size~$\tau$. This step size corresponds to the splitting as well as the temporal discretization later on. 
%
%
\subsection{Bulk and surface subproblems}\label{sect:splitting:subproblems} 
For the definition of the bulk problem, the idea is to interpret $p$ as prescribed boundary data. Moreover, we apply an index reduction technique known from DAE theory in order to include additional boundary information into the subproblem, see, e.g., \cite{HaiW96,KunM06}.

For the sake of clarity, we consider the finite element spaces introduced in~\eqref{eq:stableFEM}, for which the boundary nodes of $u$ coincide with the degrees of freedom for the Lagrange multiplier~$\lambda$. Assuming a corresponding ordering of the nodes, we conclude that~$\Bu$ has the form~$\Bu = [\, 0\ \Mla\,] \in \R^{\dofLam \times \dofOm}$, where $\Mla \in \R^{\dofLam \times \dofLam}$ is the standard (symmetric) mass matrix corresponding to the space~$\Mh$. This ordering also implies a decomposition of the bulk variable, namely 
\begin{align}
\label{eq:u1u2}
	u = \begin{bmatrix}
	u_1 \\ u_2
	\end{bmatrix}
	\qquad\text{with }
	u_1(t)\in\R^{\dofOm-\dofLam},\ u_2(t)\in\R^{\dofLam}.
\end{align}
In the same manner, the right-hand side~$\fu$ is decomposed into $\fu[1]$ and~$\fu[2]$ and the matrices~$\Mu$, $\Ku$ are decomposed into blocks $M_{ij}$, $K_{ij}$, $i,j=1,2$.

\begin{remark}
For different choices of the spatial discretization, $\Bu$ may not have this special structure. Nevertheless, by its full row-rank property, there exists an ordering of the nodes (and a corresponding decomposition of~$u$, $\fu$, $\Mu$, and~$\Ku$) such that the last $\dofLam$ columns of $\Bu$ form an invertible block, which is sufficient for the following construction.  
\end{remark}

We first introduce the {\em bulk problem}, which is obtained from~\eqref{eq:semidiscreteDAE} by removing the equation for~$p$ and considering $p$ as input in the constraint equation. This leads to the DAE 
\begin{subequations}
\label{eq:sub1:bulk_index2}
\begin{align}
	\Mu \dot u + \Ku u - \Bu^T \lambda 
	&= \fu(u), \\
	\Bu u \phantom{+ MB \lambda}
	&= \Bp p .
\end{align}
\end{subequations}
It was shown in~\cite{AltKZ22} that the direct use of this formulation within a splitting scheme does not yield satisfactory results. This is probably due to the fact that the simple inclusion of~$p$ as boundary data does not reflect enough information. Because of this, we consider a mathematically equivalent formulation of the subsystem~\eqref{eq:sub1:bulk_index2}, given by an approach called {\em minimal extension}; see~\cite{MatS93,KunM06}. Introducing another auxiliary variable $w$ in place of~$\dot u_2$, we obtain as bulk problem 
\begin{subequations}
\label{eq:sub1:bulk}
\begin{align}
	\begin{bmatrix} M_{11} & M_{12} \\ 
	M_{21} & M_{22} \end{bmatrix} 
	\begin{bmatrix} \dot u_1 \\ w \end{bmatrix}
	+ \begin{bmatrix} K_{11} & K_{12} \\ 
	K_{21} & K_{22} \end{bmatrix} 
	\begin{bmatrix} u_1 \\ u_2 \end{bmatrix}
	- \begin{bmatrix} 0 \\ \Mla \end{bmatrix} \lambda 
	&= \begin{bmatrix} \fu[1](u) \\ \fu[2](u) \end{bmatrix}, \label{eq:sub1:bulk:a} \\
	\Mla u_2 
	&= \Bp p, \label{eq:sub1:bulk:b} \\
	\Mla w 
	&= \Bp \dot p. \label{eq:sub1:bulk:c}
\end{align}
\end{subequations}
Note that we use here the introduced block structure of the matrices~$\Mu$, $\Ku$ and the right-hand side~$\fu$. Moreover, the equation $w = \dot u_2$ only appears implicitly through the two constraints~\eqref{eq:sub1:bulk:b} and~\eqref{eq:sub1:bulk:c}. As a result, the computation of~$u$ now contains information on~$p$ as well as its derivative. 

The second subsystem, the so-called {\em boundary problem}, is obtained by the second line of~\eqref{eq:semidiscreteDAE:a}, in which we replace the Lagrange multiplier by means of the first part of the equation. More precisely, we insert the equation  
\[
	- \Mla \lambda
	= \fu[2](u) - M_{22} {\dot u}_2 - K_{22} u_2 - M_{21} {\dot u}_1 - K_{21} u_1 
\]
into the second line of~\eqref{eq:semidiscreteDAE:a}, which leads to  
\begin{align}
	\Mp \dot p + \Kp p 
	= \fp(p) + \Bp^T \Mla^{-1}
	\big(\fu[2](u) - M_{22} {\dot u}_2 - K_{22} u_2 - M_{21} {\dot u}_1 - K_{21} u_1 \big).  \label{eq:sub2:boundary}
\end{align}
For practical computations, one may also calculate the Lagrange multiplier first and include this as an input. In summary, we have constructed a bulk problem~\eqref{eq:sub1:bulk} and a boundary problem~\eqref{eq:sub2:boundary}, which need to be combined in an appropriate way. 
%
%
\subsection{Inclusion of delay terms}\label{sect:splitting:delayTerms}
Since the bulk problem~\eqref{eq:sub1:bulk} still contains the boundary variable~$p$, a direct application of an implicit time discretization such as a general multi-step method will not decouple bulk and surface dynamics as desired. Following the ideas of~\cite{AltMU21,AltMU22ppt}, we motivate the construction of a decoupled scheme by the implementation of certain delay terms. In this way, the bulk problem will only depend on previous values of $p$. To make this more precise, we consider one particular possibility and replace the constraints~\eqref{eq:sub1:bulk:b} and~\eqref{eq:sub1:bulk:c} by
\begin{align}
\label{eq:sub1:bulk:bc:delayDerivative}
  \Mla u_2 = \Bp (2 p_\tau - p_{2\tau}), \qquad
  \Mla w = \Bp (2 \dot p_\tau - \dot p_{2\tau}),
\end{align}
respectively, with the delay terms defined by~$p_\tau\colon t \mapsto p(t-\tau)$. Note that this calls for a given history function for $p$ on $[-2\tau, 0]$ rather than only initial data at time $t=0$. Further note that the delay time coincides with the intended step size (and twice the step size) of the splitting scheme. The following lemma shows that such a modification only causes a perturbation of the solution by a term of order~$\tau^2$ and motivates the splitting schemes of the subsequent section. 
\begin{lemma}
\label{lem:error_continuous} 
Let $(u,p,\lambda)$ denote the solution of~\eqref{eq:semidiscreteDAE} -- or equivalently~\eqref{eq:sub1:bulk} and~\eqref{eq:sub2:boundary}. Further, let $(\tilde{u},\tilde p,\tilde \lambda)$ denote the solution of~\eqref{eq:sub1:bulk:a} and~\eqref{eq:sub2:boundary} with constraints~\eqref{eq:sub1:bulk:bc:delayDerivative}. Then, under the assumption of a sufficiently smooth solution~$\tilde p$, namely~$\tilde p \in W^{3,1}(-2\tau,T;\R^{\dofGa})$, it holds that 
\begin{align*}
	\| u(t)&-\tilde{u}(t) \|_{\Mu}^2 + \tfrac 12\, \| p(t)-\tilde{p}(t) \|_{\Mp}^2 
	+ \int_0^t 2\, \| u(s)-\tilde{u}(s) \|_{\Ku}^2 + \| p(s)-\tilde{p}(s) \|_{\Kp}^2 \ds \\
	\leq&\, \Bigg(\Big(2 \|\eta(t)\|_{\Mp}^2 + \int_0^t \|\eta(s)\|_{\Kp}^2 \ds\Big)^{1/2} + \sqrt 8 \int_0^t \|\dot \eta(s)\|_{\Mp} \ds \Bigg)^2\\
	%
	%
	\leq&\, 4\tau^4 \max_{s \in (-2\tau,t)} \|\tilde p^{(2)}(s)\|_{\Mp}^2 
	+ 4\tau^4 \int_{-2\tau}^t \|\tilde p^{(2)}(s)\|_{\Kp}^2 \ds 
	+ 12\tau^4\, \bigg( \int_{-2\tau}^t \|\tilde p^{(3)}(s)\|_{\Mp}\ds \bigg)^2. 
\end{align*}
Moreover, under the additional regularity assumption $\tilde p \in H^{3}(-2\tau,T;\R^{\dofGa})$, we obtain  
\begin{align*}
	\| u(t)-\tilde{u}(t) \|_{\Kp}^2 +\| p(t)&-\tilde{p}(t) \|_{\Kp}^2 + \int_0^t 2\, \| \dot u(s)-\dot{\tilde{u}}(s) \|_{\Mu}^2 + \| \dot p(s)-\dot{\tilde{p}}(s) \|_{\Mp}^2 \ds \\
	\leq\,& \Bigg(\Big( \int_0^t \|\dot \eta(s)\|_{\Mp}^2 \ds\Big)^{1/2} +  \sqrt 8 \int_0^t \|\dot \eta(s)\|_{\Kp} \ds \Bigg)^2\\
	\leq\,& \tfrac{10}3\tau^4 \int_{-2\tau}^t \|\tilde p^{(3)}(s)\|_{\Mp}^2 \ds 
	+ \tfrac{10}3\tau^4\, \bigg(\int_{-2\tau}^t \|\tilde p^{(3)}(s)\|_{\Kp} \ds \bigg)^2.
\end{align*}
\end{lemma}
\begin{proof}
Let the function~$\eta$ be defined as $\eta \coloneqq \tilde p - 2\tilde p_\tau + \tilde p_{2\tau}$. Then, one considers the differential equation for $(u-\tilde{u}, p-\tilde{p}, \lambda- \tilde{\lambda})$ with test functions~$(u-\tilde{u}, p-\tilde{p} + \eta)$ and their derivatives, respectively. Integration by parts, Young's inequality, and a Gronwall-type inequality as in~\cite[Th.~4.22]{Zim21} yields the upper bound
\begin{equation*}
	\Bigg(\Big(2 \|\eta(t)\|_{\Mp}^2 + \int_0^t \|\eta(s)\|_{\Kp}^2 \ds\Big)^{1/2} + \sqrt 8 \int_0^t \|\dot \eta(s)\|_{\Mp} \ds \Bigg)^2
\end{equation*}
The stated bound then follows by Taylor's theorem.
\end{proof}
\begin{remark}
The estimates of Lemma~\ref{lem:error_continuous} remain valid in the fully continuous setting (i.e., without a spatial discretization) with the associated inner products. For this, the first error estimate calls for~$\tilde{p} \in W^{3,1}(-2\tau,T;L^2(\Gamma)) \cap H^2(-2\tau,T;H^1(\Gamma))$, whereas the second estimate needs $\tilde{p} \in H^{3}(-2\tau,T;L^2(\Gamma)) \cap W^{3,1}(-2\tau,T;H^1(\Gamma))$. 
\end{remark}
With the previous lemma, we have shown that constraints including a delay provide a promising basis for the construction of second-order splitting schemes. We present three possibilities how to implement different approximations of the derivative in~\eqref{eq:sub1:bulk:c}: 
A simple first-order approximation of the form~$\dot p \approx \dot p_\tau \approx \tfrac{1}{\tau}\, (p_\tau - p_{2\tau})$ yields the constraints 
\begin{align}
\label{eq:sub1:bulk:bc:delayA}
	\Mla u_2 = \Bp (2 p_\tau - p_{2\tau}), \qquad
	\Mla w = \tfrac{1}{\tau}\, \Bp (p_\tau - p_{2\tau}).
\end{align}
%
A second-order approximation of the derivatives is obtained as follows. First, we note that~$\dot p \approx 2 \dot p_\tau - \dot p_{2\tau}$. Then, we apply the well-known BDF-$2$ formula, leading to 
\begin{align}
\label{eq:sub1:bulk:bc:delayC}
\Mla u_2 = \Bp (2 p_\tau - p_{2\tau}), \qquad
\Mla w = \tfrac{1}{2\tau}\, \Bp (6p_\tau - 11p_{2\tau} + 6p_{3\tau} - p_{4\tau}). 
\end{align}
Note that the improvement in accuracy comes at the price of additional delay terms, i.e., the bulk problem now needs information of~$p$ at times~$t-\tau$, $t-2\tau$, $t-3\tau$, and~$t-4\tau$.  

Finally, we introduce a second-order approximation with only three delay terms. Starting point is again~$\dot p \approx 2 \dot p_\tau - \dot p_{2\tau}$. But now, $\dot p_{2\tau}$ is approximated by a central difference, i.e., $\dot p_{2\tau} \approx \tfrac{1}{2\tau} (p_{\tau} - p_{3\tau})$. This then leads to the constraints 
\begin{align}
\label{eq:sub1:bulk:bc:delayB}
	\Mla u_2 = \Bp (2 p_\tau - p_{2\tau}), \qquad
	\Mla w = \tfrac{1}{2\tau}\, \Bp (5p_\tau - 8p_{2\tau} + 3p_{3\tau}).
\end{align}
In order to obtain a fully discrete splitting scheme, a temporal discretization of the bulk and boundary problem is necessary. In this paper, we focus on the approach given in~\eqref{eq:sub1:bulk:bc:delayB}, which will be analyzed in Section~\ref{sect:convergence}. The schemes resulting from the constraints~\eqref{eq:sub1:bulk:bc:delayA} and~\eqref{eq:sub1:bulk:bc:delayC} are considered numerically in Section~\ref{sect:numerics}. 
%
%
\subsection{A bulk--surface splitting scheme of second order} 
\label{sect:splitting:secondOrderScheme}
We close this section with the presentation of an actual numerical scheme separating bulk and surface dynamics. For this, we consider~\eqref{eq:sub1:bulk:a} with constraints~\eqref{eq:sub1:bulk:bc:delayB} as bulk problem and~\eqref{eq:sub2:boundary} as boundary problem. We discretize both subproblems in time by the well-known BDF-$2$ formula. Introducing the short notion
\begin{align*}
	\DBDF x^{n+3} 
	&\coloneqq \tfrac 1{2\tau}\, (3 x^{n+3} - 4 x^{n+2} + x^{n+1}), \\
	\Dalt x^{n+2} 
	&\coloneqq \tfrac 1{2\tau}\, (5 x^{n+2} - 8 x^{n+1} + 3x^n),
\end{align*}
we obtain the following bulk--surface splitting scheme: 
Given previous approximations~$u^{n+1}$, $u^{n+2}$, $p^n$, $p^{n+1}$, and $p^{n+2}$, first seek~$(u^{n+3}_1, u^{n+3}_2, \lambda^{n+3}, w^{n+3})$ by
\begin{subequations}
\label{eq:splittingScheme:fullyDiscrete}
\begin{align}
	M_{11} \DBDF u^{n+3}_1 + K_{11} u^{n+3}_1 \hphantom{\ - \Mla \lambda^{n+3}} 
	&= f_1^{n+3} - M_{12} w^{n+3} - K_{12} u^{n+3}_2 \label{eq:splittingScheme:fullyDiscrete:a} \\
	M_{21} \DBDF u^{n+3}_1 + K_{21} u^{n+3}_1 - \Mla \lambda^{n+3} 
	&= f_2^{n+3} - M_{22} w^{n+3} - K_{22} u^{n+3}_2 \label{eq:splittingScheme:fullyDiscrete:b} \\
	\Mla u_2^{n+3} 
	&= \Bp (2p^{n+2} - p^{n+1}) \label{eq:splittingScheme:fullyDiscrete:c} \\
	\Mla w^{n+3} 
	&= \Bp \Dalt p^{n+2} \label{eq:splittingScheme:fullyDiscrete:d}
\end{align}
and then $p^{n+3}$ by
\begin{align}	
	\Mp \DBDF p^{n+3} + \Kp p^{n+3} + \Bp^T \lambda^{n+3} 
	&= \fp^{n+3}. \label{eq:splittingScheme:fullyDiscrete:e}
\end{align}
\end{subequations}
Here, the right-hand sides are given by 
\begin{equation*}
	\fu^{n+3} 
	= \fu(t^{n+3},u^{n+3}),\qquad 
	\fp^{n+3}
	= \fp(t^{n+3},p^{n+3}). 
\end{equation*}
We would like to emphasize that this scheme not only decouples bulk and surface dynamics but allows the following sequence of solution steps: 
First compute~$u_2^{n+3}$ and~$w^{n+3}$ by the constraints~\eqref{eq:splittingScheme:fullyDiscrete:c} and~\eqref{eq:splittingScheme:fullyDiscrete:d}. Second, derive~$u_1^{n+3}$ by~\eqref{eq:splittingScheme:fullyDiscrete:a} and~$\lambda^{n+3}$ by~\eqref{eq:splittingScheme:fullyDiscrete:b}. Then, finally, use~\eqref{eq:splittingScheme:fullyDiscrete:e} for the update of the boundary variable~$p^{n+3}$. 
%
%
\section{Convergence Analysis}
\label{sect:convergence}
This section is devoted to the analysis of the $3$-step bulk--surface splitting method~\eqref{eq:splittingScheme:fullyDiscrete}, which is based on the constraints given in~\eqref{eq:sub1:bulk:bc:delayB}. Since we consider a multi-step method, we need to assume given data for the first two time points. In practice, these have to be computed by another time integration scheme of appropriate order. To prove second-order convergence, we will first introduce an auxiliary method and prove stability and convergence of this scheme. Afterwards, we translate this result to the proposed splitting method. Within this section, we concentrate on the linear case and comment on needed adjustments for the semi-linear case later on. 

For the upcoming analysis, we make the following assumption on the spatial discretization. 
\begin{assumption}[Spatial discretization]\label{ass:mesh}~
\begin{enumerate}[label=\alph*)]
	\item The triangulation $\calT_{\Omega}$ is quasi-uniform.\label{ass:mesh_a}
	\item $\Vh|_\Gamma = \Mh$, i.e., every discrete function $u|_\Gamma \in \Vh|_\Gamma$ is an element of the finite element space $\Mh$ and vice versa. \label{ass:mesh_b}
\end{enumerate}
\end{assumption} 
\begin{remark}
Recall that the proposed finite element scheme~\eqref{eq:stableFEM} satisfies the second bullet of Assumption~\ref{ass:mesh}. The main purpose of this point is to simplify the implementation as it implies the special structure~$B_u = [0\ \Mla]$. Theoretically, it is sufficient to assume an inf--sup condition of the form 
\begin{equation*}
	\adjustlimits \inf_{\lambda \in \Mh\setminus\{0\}} \sup_{u \in \Vh, u|_\Gamma \neq 0} \frac{\int_\Gamma u \lambda \dx}{\|u|_\Gamma\|_{L^2(\Gamma)}\|\lambda\|_{L^2(\Gamma)}} 
	\geq \beta 
	> 0.
\end{equation*}
\end{remark}
%
%
\subsection{Convergence of an auxiliary splitting scheme}\label{sect:convergence:auxiliary}
As indicated in the end of Section~\ref{sect:splitting}, the variables~$u_2$ and~$w$ are thoroughly determined by values of $p$ at previous time points. Hence, the proposed method~\eqref{eq:splittingScheme:fullyDiscrete} can be reduced to a time stepping scheme for~$u_1$, $p$, and~$\lambda$. For this, we insert~$\Mla^{-1}\Bp (2p^{n+2} - p^{n+1})$ for $u_2^{n+3}$ and~$\Mla^{-1}\Bp\Dalt p^{n+2}$ for~$w^{n+3}$ in the first two equations of~\eqref{eq:splittingScheme:fullyDiscrete}. 
In preparation for the convergence analysis, we further modify the splitting scheme by replacing equation~\eqref{eq:splittingScheme:fullyDiscrete:c} by $\Mla u_2^{n+3} = \Bp p^{n+3}$. This then defines a closely related numerical scheme, which only differs from~\eqref{eq:splittingScheme:fullyDiscrete} in the variable~$u_2$. In the following, we will indicate this modified solution by~$\hat{u}^n_2$ (and $\hat{u}^n$ correspondingly, although $u^n_1$ does not change). To summarize, the tuple~$(u_1^{n+3}, \hat{u}_2^{n+3},p^{n+3}, \lambda^{n+3})$ satisfies 
\begin{subequations}
	\label{eqn:BDF_sequential}
	\begin{align}
		M_{11} \DBDF u^{n+3}_1 + K_{11} u^{n+3}_1 \hphantom{\, + \Mla^T \lambda^{n+3}}  
		&= f_1^{n+3} - M_{12} \Mla^{-1}\Bp\Dalt p^{n+2} \label{eqn:BDF_sequential_a}\\ 
		&\hspace{2.1cm}- K_{12} \Mla^{-1}\Bp (2p^{n+2} - p^{n+1}), \notag\\ 
		M_{21} \DBDF u^{n+3}_1 + K_{21} u^{n+3}_1 - \Mla \lambda^{n+3} 
		&= f_2^{n+3} - M_{22} \Mla^{-1}\Bp\Dalt p^{n+2} \label{eqn:BDF_sequential_b} \\
		&\hspace{2.1cm}- K_{22} \Mla^{-1}\Bp (2p^{n+2} - p^{n+1}), \notag \\
		\Mp \DBDF p^{n+3} + \Kp p^{n+3} + \Bp^T \lambda^{n+3} 
		&= \fp^{n+3}, \label{eqn:BDF_sequential_c} \\
		\Mla \hat u_2^{n+3} 
		&= \Bp p^{n+3}. \label{eqn:BDF_sequential_d}
	\end{align}	
\end{subequations}
The main advantage of the auxiliary scheme~\eqref{eqn:BDF_sequential} is that the Lagrange multiplier~$\lambda^{n+3}$ vanishes if we test the equations~\eqref{eqn:BDF_sequential_a}-\eqref{eqn:BDF_sequential_c} with $[(\hat u^{n+3})^T,\,(p^{n+3})^T]^T$. This will be exploited in the upcoming analysis.

In order to show convergence of the auxiliary integration scheme~\eqref{eqn:BDF_sequential}, we first prove stability of the scheme. For this, we introduce the notation  
\begin{equation*}
	M_{\bullet 2} 
	\coloneqq \begin{bmatrix}
		M_{12}\\ M_{22}
	\end{bmatrix},\qquad
	K_{\bullet 2} 
	\coloneqq \begin{bmatrix}
	K_{12}\\ K_{22}
	\end{bmatrix}.
\end{equation*}
%
%
\subsubsection{Stability}
The key ingredient to show stability is an estimate, which connects the~$L^2$-norms of the bulk and of the surface variable.
\begin{lemma}
\label{lem:cM}
Consider polynomial finite element spaces $\Vh$, $\Mh$, $\Qh$ satisfying Assumption~\ref{ass:mesh} and discrete functions~$u\in \Vh$, $p \in \Qh$ satisfying~$\Bu u = \Bp p$. Then there exist constants $c_M, c_K > 0$, only depending on the uniformity parameter of the underlying mesh~$\calT_{\Omega}$ but not on $h$, such that  
\begin{equation*}
	\|u_2 \|_{M_{22}}^2 
	\leq c_M h\, \|p\|_{\Mp}^2 \quad \text{ and } \quad \|u_2 \|_{K_{22}}^2 
	\leq c_K h^{-1} \|p\|_{\Mp}^2.
\end{equation*}
\end{lemma}
\begin{proof}
Due to $\Mla u_2 =\Bu u = \Bp p$, the function~$u_2|_\Gamma$ is the~$L^2(\Gamma)$-best approximation of~$p$ in $\Vh|_\Gamma = \Mh$. This implies 
\begin{equation*}
	\|u_2 \|_{M_{22}}^2 
	\leq \c_M h\, \|u_2|_\Gamma \|_{L^2(\Gamma)}^2 
	\leq \c_M h\, \|p\|_{\Mp}^2. 
\end{equation*}
The first inequality is proven in \cite[Lem.~3.1]{AltKZ22} for $\calP_1$-elements and can be generalized to polynomials of arbitrary degree in the same manner. The second bound follows by the inverse estimate for finite elements \cite[Ch.~II.6.8]{Bra07}.
\end{proof}
Besides $\DBDF$ and $\Dalt$ introduced in Section~\ref{sect:splitting:secondOrderScheme}, we denote the difference of two consecutive iterates by 
\begin{equation*}
	\EBDF x^{n+1} 
	\coloneqq x^{n+1} - x^n.
\end{equation*}
With this, we can express the difference of the two discrete derivatives as
\begin{equation}\label{eqn:difference_derivatives}
	2\tau\, \big(\DBDF x^{n+3}-\Dalt x^{n+3}\big) 
	= 3x^{n+3} - 9x^{n+2} + 9x^{n+1} - 3x^{n} = 3 \EBDF^3 x^{n+3}.
\end{equation} 
Moreover, we obtain the following lemma. 
\begin{lemma}
\label{lem:testing_BDF_with_difference_help}
Given a sequence $\{x^n\}_{n\in N}$ and a symmetric and positive definite matrix~$M$, we have 
\begin{equation}
	\label{eqn:testing_BDF_with_difference_help}
	2\tau\, (\EBDF x^{n+3}) \cdot (M \DBDF x^{n+3}) 
	= \tfrac 52\|\EBDF x^{n+3}\|_M^2 - \tfrac 12 \|\EBDF x^{n+2}\|_M^2 + \tfrac 12 \|\EBDF^2 x^{n+3}\|_M^2.
\end{equation}
\end{lemma}
\begin{proof}
The stated equality follows by
\begin{align*}
	(a-b)(3a-4b+c) 
	&= 2\,(a-b)^2 + (a-b)\big((a-b) - (b-c) \big)\\
	&= 2\,(a-b)^2 + \tfrac 12 \big( (a-b)^2-(b-c)^2+(a-2b+c)^2 \big),
\end{align*}
where we assign $a=M^{\sfrac 12}x^{n+3}$, $b=M^{\sfrac 12}x^{n+2}$, and $c=M^{\sfrac 12}x^{n+1}$.
\end{proof}
We can now prove the stability of the auxiliary splitting scheme. 
\begin{proposition}[Stability of scheme~\eqref{eqn:BDF_sequential}]
\label{prop:BDF_stability}
Let $u^{1}, u^2 \in \Vh$, $p^{0}, p^{1}, p^2\in \Qh$ be given initial data with consistent $u^2, p^2$, i.e., $\Bu u^2 = \Bp p^2$, and $\{\hat{u}^n, p^n, \lambda^n\}_{n\in \N_{\geq 3}} \subset \Vh\times \Qh \times \Mh$ the sequence determined by~\eqref{eqn:BDF_sequential}. Then, it holds that
\begin{align*}
	&\quad\|\hat u^{n+3}\|_{\Ku}^2 + \|p^{n+3}\|_{\Kp}^2 + \tau \sum_{k=3}^{n+3} \| \tfrac 1 \tau \EBDF \hat u^{k}\|^2_{\Mu}  + \| \tfrac 1 \tau \EBDF p^{k}\|^2_{\Mp} + \tfrac 1 2 \| \tfrac 1 \tau \EBDF^2 \hat u^{k}\|^2_{\Mu}\\
	&\quad\qquad + \tau  \sum_{k=3}^{n+3} \big(\tfrac 1 2 - 18  c_M h  - c_{K} \tau h^{-1}\big)\, \|\tfrac{1}{\tau} \EBDF^2 p^{k} \|_{\Mp}^2 + \sum_{k=3}^{n+3} \| \EBDF p^{k}\|^2_{\Kp}\\
	&\leq \|u^{2}\|_{\Ku}^2 + \|p^{2}\|_{\Kp}^2 + \tfrac \tau 2\, \|\tfrac 1 \tau \EBDF u^2 \|_{\Mu}^2 + \tfrac \tau 2\, \|\tfrac 1 \tau \EBDF p^2\|_{\Mp}^2 +  9 c_M \tau h\, \|\tfrac 1 \tau \EBDF^2p^{2}\|_{\Mp}^2\\* 
	&\quad\qquad + 2\tau \sum_{k=3}^{n+3} \|\fu^{k}\|_{\Mu^{-1}}^2 + \|\fp^{k} \|_{\Mp^{-1}}^2
\end{align*}
for all $n\ge0$. Therein, $c_M$ and $c_K$ denote the constants from Lemma~\ref{lem:cM}. 
\end{proposition}
\begin{proof}
With the identity~\eqref{eqn:difference_derivatives} it follows that the scheme~\eqref{eqn:BDF_sequential} can be rewritten as a perturbed BDF method, namely 
\begin{multline}\label{eqn:BDF_perturbed}
	\begin{bmatrix}
	\Mu & 0 \\
	0  & \Mp
	\end{bmatrix}
	\begin{bmatrix}
	\DBDF  \hat u^{n+3}\\
	\DBDF  p^{n+3}
	\end{bmatrix}
	+
	\begin{bmatrix}
	\Ku & 0\\
	0 & \Kp
	\end{bmatrix}
	\begin{bmatrix}
	 \hat u^{n+3}\\
	 p^{n+3}
	\end{bmatrix}
	+\begin{bmatrix}
	-\Bu^T\\
	\Bp^T
	\end{bmatrix}
	 \lambda^{n+3}\\
	-\begin{bmatrix}
	M_{\bullet 2}\\
	0
	\end{bmatrix}
	\tfrac 3 {2\tau} \EBDF^3 \hat u_2^{n+3} 
	-\begin{bmatrix}
	K_{\bullet 2}\\
	0
	\end{bmatrix}
	\EBDF^2 \hat u_2^{n+3}=
	\begin{bmatrix}
	\fu^{n+3}\\
	\fp^{n+3}
	\end{bmatrix} 
\end{multline}
together with the constraint~\eqref{eqn:BDF_sequential_d}, i.e., $\Mla \hat u_2^{n+3} = \Bu \hat u^{n+3} = \Bp p^{n+3}$. 
Testing equation~\eqref{eqn:BDF_perturbed} with $[(\EBDF \hat u^{n+3})^T,\, (\EBDF p^{n+3})^T]^T$ and using the identity~\eqref{eqn:testing_BDF_with_difference_help}, we obtain the estimate 
\begin{align*}
	&\tfrac 1\tau\, \|\EBDF \hat u^{n+3}\|^2_{\Mu} + \tfrac 1 \tau\, \|\EBDF p^{n+3}\|^2_{\Mp} + \tfrac 1 {4\tau}\, \|\EBDF^2 \hat u^{n+3}\|_{\Mu}^2 + \tfrac 1 {4\tau}\, \|\EBDF^2 p^{n+3}\|_{\Mp}^2\\
	&\quad+ \tfrac{1}{4\tau}\, \big(\|\EBDF \hat u^{n+3}\|^2_{\Mu} - \|\EBDF \hat u^{n+2}\|^2_{\Mu}\big) + \tfrac{1}{4\tau}\, \big(\|\EBDF p^{n+3}\|^2_{\Mp} - \|\EBDF p^{n+2}\|^2_{\Mp} \big)\\
	&\quad+ \tfrac{1}{2}\, \big(\|\hat u^{n+3}\|^2_{\Ku} - \|\hat u^{n+2}\|^2_{\Ku} + \|\EBDF \hat u^{n+3}\|^2_{\Ku}\big) + \tfrac{1}{2}\, \big(\|p^{n+3}\|^2_{\Kp} - \|p^{n+2}\|^2_{\Kp} + \|\EBDF p^{n+3}\|^2_{\Kp}\big)\\
	&\qquad\quad= \big\langle \EBDF \hat u^{n+3}, \big( \tfrac{3}{2\tau} M_{\bullet 2}\EBDF +  K_{\bullet 2}\big) \EBDF^2 \hat u_2^{n+3} \big\rangle + \langle \fu^{n+3},\EBDF \hat u^{n+3}\rangle + \langle \fp^{n+3},\EBDF p^{n+3}\rangle.
\end{align*}
Note that the terms involving~$\lambda^{n+3}$ vanish due to equation~\eqref{eqn:BDF_sequential_d}. Using Young's inequality, we can bound the first term of the right-hand side by 
\begin{align*}
	\big\langle \EBDF \hat u^{n+3}, \tfrac 3 {2\tau} M_{\bullet 2}\EBDF^3 \hat u_2^{n+3} \big\rangle 
	&\leq \tfrac 1{4\tau}\, \|\EBDF \hat u^{n+3}\|_{\Mu}^2 + c_M h \tfrac{9}{2\tau}\, \big(\|\EBDF^2 p^{n+2}\|_{\Mp}^2+\|\EBDF^2 p^{n+3}\|_{\Mp}^2 \big),\\
	\big\langle \EBDF \hat u^{n+3}, K_{\bullet 2} \EBDF^2 \hat u_2^{n+3} \big\rangle
	&\leq \tfrac 12\, \|\EBDF \hat u^{n+3}\|^2_{\Ku} + \tfrac{1}{2} c_{K}h^{-1} \| \EBDF^2 p^{n+3}\|_{\Mp}^2.
\end{align*}
Finally, the stated bound follows by summing up the resulting inequality from $0$ to $n$. 
\end{proof}
We would like to emphasize that Proposition~\ref{prop:BDF_stability} guarantees stability of scheme~\eqref{eqn:BDF_sequential} as long as 
\[
  \tfrac 12 - 18  c_M h  - c_{K} \tau h^{-1}
  \ge 0.
\]
This motivates the following assumption on the time step size.
\begin{assumption}[Weak CFL condition]\label{ass:cfl}
The mesh width~$h$ is sufficiently small and the time step size~$\tau$ satisfies the CFL-type condition~$\tau\lesssim h$, i.e., there exists a generic constant~$C$, independent of $\tau$ and $h$, such that~$\tau \leq C\, h$. 
\end{assumption} 
\begin{remark}
The requirement of Assumption~\ref{ass:cfl} is weak in the sense that classical CFL conditions for parabolic problems read~$\tau \lesssim h^2$, see, e.g., \cite[p.~318]{KnaA03}.
\end{remark}
%
%
\subsubsection{Second-order convergence}
In this subsection, we prove that the auxiliary scheme~\eqref{eqn:BDF_sequential} is convergent of order two. For this, we will apply the following residual estimates, which directly result from Taylor's theorem.  
\begin{lemma}
\label{lem:error_DandE}
Let the scalar function $r$ be sufficiently regular. Then, it holds that 
\begin{gather*}
	\big|\DBDF r(t^{n+2}) - \dot r(t^{n+2})\big|^2 
	\leq \frac{\tau^3}{2} \int_{t^n}^{t^{n+2}} |\dddot{r}(s)|^2 \ds,\\
	\big|\EBDF^2 r(t^{n+2})\big|^2 
	\leq \frac{2\,\tau^3}{3} \int_{t^n}^{t^{n+2}} |\ddot{r}(s)|^2 \ds,\qquad
	\big|\EBDF^3 r(t^{n+3})\big|^2 
	\leq \frac{9\,\tau^5}{8} \int_{t^{n}}^{t^{n+3}} |\dddot{r}(s)|^2 \ds.
\end{gather*}
\end{lemma}
To prove convergence, we first consider the discrete norms corresponding to the spaces $L^\infty(0,T;H^1(\Omega)\times H^1(\Gamma))$. Note that these norms are stronger than the usual ones corresponding to~$L^\infty(0,T;L^2(\Omega)\times L^2(\Gamma))$ and $L^2(0,T;H^1(\Omega)\times H^1(\Gamma))$, respectively. The errors in these norms will be studied afterwards. 
\begin{theorem}\label{th:BDF_convergence_order_strong_norm}
Suppose Assumptions~\ref{ass:mesh} and~\ref{ass:cfl} are satisfied and let the solution $(u,p,\lambda)$ of the semi-discrete problem~\eqref{eq:semidiscreteDAE} be sufficiently regular. If the initial data satisfies 
\begin{align*}
  \| u^k - u(t^k)\|_{\Mu}^2 + \tau\, \| u^k - u(t^k)\|_{\Ku}^2 
  + \| p^k - p(t^k)\|_{\Mp}^2 + \tau\, \| p^k - p(t^k)\|_{\Kp}^2 
  \lesssim \tau^{4}
\end{align*}
for $k=0,1,2$ with consistent $u^2,p^2$, then 
\begin{equation*}
	\| \hat u^{n} - u(t^{n})\|_{\Ku}^2 + \| p^{n} - p(t^{n})\|_{\Kp}^2 
	\lesssim \tau^3
\end{equation*}
holds for every $n\ge0$. 
\end{theorem}
\begin{proof}
Recall that $(u,p,\lambda)$ denotes the exact solution and $(\hat u^n,p^n,\lambda^n)$ the approximation determined by~\eqref{eqn:BDF_sequential}. To shorten notation, we introduce the differences~$e_{\hat u}^k \coloneqq u(t^k)-\hat u^k$, $e_p^k \coloneqq p(t^k)-p^k$, and $e_\lambda^k \coloneqq \lambda(t^k)-\lambda^k$. Then, $(e_{\hat u}^{n},e_p^{n},e_\lambda^{n})$ satisfies the perturbed BDF system~\eqref{eqn:BDF_perturbed}, where the right-hand side is replaced by 
\begin{equation*}
	-\begin{bmatrix}
	\frac{3}{2\tau}M_{\bullet 2}\EBDF + K_{\bullet 2}\\
	0
	\end{bmatrix}
	\EBDF^2 u_2(t^{n+3})
	+
	\begin{bmatrix}
	\Mu & 0 \\
	0 & \Mp
	\end{bmatrix}
	\begin{bmatrix}
	\DBDF u(t^{n+3}) - \dot u(t^{n+3})\\
	\DBDF p(t^{n+3}) - \dot p(t^{n+3})
	\end{bmatrix}.
\end{equation*}
Now, we follow the steps of the proof of Proposition~\ref{prop:BDF_stability} and apply the test function $[(\EBDF e_{\hat u}^{n+3})^T,\, (\EBDF e_p^{n+3})^T]^T$. In combination with Lemma~\ref{lem:error_DandE}, this yields  
\begin{multline}
	\label{eqn:BDF_convergence_order_strong_norm_help}
	\|e_{\hat u}^{n+3}\|_{\Ku}^2 + \|e_p^{n+3}\|_{\Kp}^2 + \tau \sum_{k=3}^{n+3} \|\tfrac 1 \tau \EBDF e_{\hat u}^{k}\|^2_M +  \|\tfrac 1 \tau \EBDF e_p^{k}\|^2_{\Mp} + \big(\tfrac 1 2 - 18  c_M h  - c_{K} \tau h^{-1}\big)\, \|\tfrac 1 \tau \EBDF^2 e_p^{k}\|^2_{\Mp}\\
	\lesssim \|e_{\hat u}^2\|_{\Ku}^2 + \|e_p^2\|_{\Kp}^2 + \tau\, \|\tfrac 1 \tau \EBDF e_{\hat u}^2\|_{\Mu}^2 + \tau\,\|\tfrac 1 \tau \EBDF e_p^2 \|_{\Mp}^2 + \tfrac{h}{\tau}\,  \|\EBDF^2e_p^2\|_{\Mp}^2\\*
	+ \tau^4 h \int_{t^{0}}^{t^{n+3}} \|\dddot p\|_{\Mp}^2 \ds
	+ \tau^3 \int_{t^{1}}^{t^{n+3}} \|\ddot u_2\|_{K_{22}}^2 + \tau\, \big( \|\dddot u\|_{\Mu}^2 + \|\dddot p\|_{\Mp}^2 \big)\ds. 
\end{multline}
The stated estimate then follows by the assumption on the initial data.
\end{proof}
\begin{remark}
Following the lines of the proof of Theorem~\ref{th:BDF_convergence_order_strong_norm}, one can show that the error in the derivatives measured in the discrete $L^2(0,T;L^2(\Omega)\times L^2(\Gamma))$-norm is of order $1.5$. Here, one uses the identity  
\begin{multline*}
	2\tau\, (\EBDF x^{n+3})\cdot (M \DBDF x^{n+3}) \\
	= 2\, \|\EBDF x^{n+3}\|_M^2 - \tfrac 13\, \|\EBDF x^{n+2}\|_M^2 + \tfrac 14\, \|\EBDF^2 x^{n+3}\|_M^2 + \tfrac 13\, \|\tau \DBDF x^{n+3}\|_M^2.
\end{multline*}
This then implies $\tau \sum_{k=1}^{n+3} \|\lambda^k-\lambda(t^k)\|_{H^{-\sfrac 1 2}(\Gamma)}^2 \lesssim \tau^3$.
\end{remark}
Next, we prove the second-order rate in weaker norms. For this, we use the well-known identity, see e.g., \cite[p.~203]{Emm04},
\begin{multline}
	\label{eqn:testing_BDF}
  4\tau\, x^{n+3}\cdot (M\DBDF x^{n+3}) 
  = \|x^{n+3}\|_M^2 - \|x^{n+2}\|_M^2\\
  + \|2x^{n+3} - x^{n+2}\|_M^2 - \|2x^{n+2} - x^{n+1}\|_M^2 + \|x^{n+3} - 2x^{n+2} + x^{n+1}\|_M^2.
\end{multline}
\begin{theorem}\label{th:BDF_convergence_order_weak_norm}	
Suppose the assumptions of Theorem~\ref{th:BDF_convergence_order_strong_norm} are satisfied. Moreover, let the consistent initial data satisfy
\begin{align*}
\| u^k - u(t^k)\|_{\Mu}^2 + \tau\, \| u^k - u(t^k)\|_{\Ku}^2 
+ \| p^k - p(t^k)\|_{\Mp}^2 + \tau\, \| p^k - p(t^k)\|_{\Kp}^2 
\lesssim \tau^{5}
\end{align*}
for $k=0,1,2$ with consistent $u^2,p^2$. Then the error estimate
\begin{equation*}
	\| \hat u^{n} - u(t^{n})\|_{\Mu}^2 + \| p^{n} - p(t^{n})\|_{\Mp}^2 + \tau \sum_{k=1}^n \Big( \| \hat u^{k} - u(t^{k})\|_{\Ku}^2 + \| p^{k} - p(t^{k})\|_{\Kp}^2 \Big)
	\lesssim \tau^4 
\end{equation*}
holds for every $n\ge0$.
\end{theorem}
\begin{proof}
In contrast to the proof of Theorem~\ref{th:BDF_convergence_order_strong_norm}, we now consider~$4\tau\, [(e_{\hat u}^{n+3})^T, (e_p^{n+3})^T]^T$ as test function in~\eqref{eqn:BDF_perturbed}. This then yields the bound
\begin{align*}
	&\quad\|e_{\hat u}^{n+3}\|^2_{\Mu} - \|e_{\hat u}^{n+2}\|^2_{\Mu} + \|\EBDF^2 e_{\hat u}^{n+3}\|^2_{\Mu} +\|2e_{\hat u}^{n+3}-e_{\hat u}^{n+2}\|^2_{\Mu} - \|2e_{\hat u}^{n+2}-e_{\hat u}^{n+1}\|^2_{\Mu}\\*
	&\qquad + \|e_p^{n+3}\|^2_{\Mp} - \|e_p^{n+2}\|^2_{\Mp} + \|\EBDF^2 e_p^{n+3}\|^2_{\Mp} +\|2e_p^{n+3}-e_p^{n+2}\|^2_{\Mp} - \|2e_p^{n+2}-e_p^{n+1}\|^2_{\Mp}\\
	&\qquad + 4\tau\, \|e_{\hat u}^{n+3}\|_{\Ku}^2  + 4\tau\, \|e_p^{n+3}\|_{\Kp}^2\\*
	&= 4\tau\, \big\langle e_{\hat u}^{n+3}, (M_{\bullet 2} \tfrac{3}{2\tau} \EBDF^3 + K_{\bullet 2} \EBDF^2) (u_2(t^{n+3})-e_{\hat u_2}^{n+3}) \big\rangle\\*
	&\qquad + 4\tau\, \big\langle \DBDF u(t^{n+3})-\dot u(t^{n+3}), \Mu e_{\hat u}^{n+3}\big\rangle + 4\tau\, \big\langle \DBDF p(t^{n+3})-\dot p(t^{n+3}), \Mp e_p^{n+3}\big\rangle \\
	&\le \|e_{\hat u}^{n+3}\|_{\Mu} \Big(6\sqrt{c_M h}\, \big(\|\EBDF^3 e_p^{n+3}\|_{\Mp} +\|\EBDF^3 p(t^{n+3})\|_{\Mp}\big) + 4\tau\, \|\DBDF u(t^{n+3})-\dot u(t^{n+3}) \|_{\Mu} \Big)\\
	&\qquad + 4\tau\, \|e_p^{n+3}\|_{\Mp} \|\DBDF p(t^{n+3})-\dot p(t^{n+3}) \|_{\Mp} + 8\tau\, \| \EBDF^2 u_2(t^{n+3})\|_{K_{22}}^2\\
	&\qquad + \tau\, \| e_{\hat u}^{n+3}\|_{\Ku}^2 +  c_K \tau h^{-1} \| \EBDF^2 e_p^{n+3}\|^2_{\Mp}
\end{align*}
Summing up from $0$ to $n$ and using \cite[Lem.~8.13]{Zim21} as well as $\sqrt a+\sqrt b \leq \sqrt 2 \sqrt{a+b}$, we derive the estimate
\begin{align*}
	&\quad \|e_{\hat u}^{n+3}\|^2_{\Mu} + \|e_p^{n+3}\|^2_{\Mp} + \tau \sum_{k=1}^{n+3} \|e_{\hat u}^{k}\|_{\Ku}^2 + \|e_p^{k}\|_{\Kp}^2\\
	&\lesssim \bigg(\bigg[\|e_{\hat u}^{2}\|^2_{\Mu} +  \|2e_{\hat u}^{2}-e_{\hat u}^{1}\|^2_{\Mu} +\|e_p^{2}\|^2_{\Mp} +  \|2e_p^{2}-e_p^{1}\|^2_{\Mp} + \tau \sum_{k=3}^{n+3} \| \EBDF^2 u_2(t^{k})\|_{K_{22}}^2 \bigg]^{\sfrac 1 2} \\*
	&\qquad + \tau \sum_{k=3}^{n+3} \Big( \|\DBDF u(t^{k})-\dot u(t^{k}) \|_{\Mu} + \|\DBDF p(t^{k})-\dot p(t^{k}) \|_{\Mp} \Big) \\
	&\qquad + \sqrt{h} \sum_{k=3}^{n+3}  \|\EBDF^3 p(t^{k})\|_{\Mp} + \|\EBDF^3 e_p^{k}\|_{\Mp} \bigg)^2\\
	&\lesssim \|e_{\hat u}^{2}\|^2_{\Mu} +  \|2e_{\hat u}^{2}-e_{\hat u}^{1}\|^2_{\Mu} +\|e_p^{2}\|^2_{\Mp} +  \|2e_p^{2}-e_p^{1}\|^2_{\Mp} + \tau^4 \int_{t^{1}}^{t^{n+3}} \|\ddot{u}_2(s)\|_{K_{22}}^2 \ds \\
	&\qquad + \tau^4\, t^{n+3} \bigg( \int_{t^{1}}^{t^{n+3}}  \|\dddot{u}(s)\|_{\Mu}^2 \ds +  (1+h) \int_{t^{0}}^{t^{n+3}}  \|\dddot{p}(s)\|_{\Mp}^2\ds + h\, \sum_{k=3}^{n+3} \tau^{-5} \|\EBDF^3 e_p^{k}\|_{\Mp}^2\bigg).
\end{align*}
Since the error of the initial data is of order~$\tau^{5}$ by assumption, it remains to show that also the last term~$\sum_{k=3}^{n+3} \tau^{-5} \|\EBDF^3 e_p^{k}\|_{\Mp}^2$ is bounded or, equivalently, $\tau \sum_{k=3}^{n+3} \|\frac 1 \tau \EBDF^3 e_p^{k}\|_{\Mp}^2 \lesssim \tau^4$. For the sake of readability, we move this part to Appendix~\ref{app:new}. 
\end{proof}
%
%
%
\subsection{Convergence of the splitting scheme~\eqref{eq:splittingScheme:fullyDiscrete}}\label{sect:convergence:original}
After dealing with the auxiliary scheme~\eqref{eqn:BDF_sequential} in the previous subsection, we now turn to the original splitting scheme~\eqref{eq:splittingScheme:fullyDiscrete}. As we have already discussed in the beginning of this section, the two schemes only differ in the approximations of $u$ (more precisely~$u_2$). For this, we observe that 
\begin{align*}
	\|u(t^n) - u^n\|^2_{\Mu} 
	&\lesssim \|e_{\hat u}^n\|^2_{\Mu} + \|\EBDF^2 \hat{u}_2^n\|^2_{M_{22}}\\
	&\lesssim \|e_{\hat u}^n\|_{\Mu}^2 + h\, \|\EBDF^2 p(t^n)\|_{\Mp}^2 + h\, \|\EBDF^2 e_{p}^n\|_{\Mp}^2\\
	&\leq \|e_{\hat u}^n\|_{\Mu}^2 + h\, \|\EBDF^2 p(t^n)\|_{\Mp}^2 + \tau h \sum_{k=1}^n \tau\, \|\tfrac 1 \tau \EBDF^2 e_{p}^k\|_{\Mp}^2.
\end{align*} 
By Theorem~\ref{th:BDF_convergence_order_weak_norm}, Lemma~\ref{lem:error_DandE}, and estimate~\eqref{eqn:BDF_convergence_order_strong_norm_help}, the right-hand side is of order $\tau^4$. Similarly, one shows that 
\begin{equation*}
	\tau \sum_{k=1}^n \|u(t^k) - u^k\|_{\Ku}^2 
	\lesssim \tau^4.
\end{equation*}
These two estimates prove the second-order convergence of the splitting scheme~\eqref{eq:splittingScheme:fullyDiscrete}. We summarize this observation in the following theorem.
\begin{theorem}[Second-order convergence]
\label{th:BDF_original}
Let Assumptions~\ref{ass:mesh} and~\ref{ass:cfl} be satisfied. Further assume that the solution $(u,p,\lambda)$ of the semi-discrete problem~\eqref{eq:semidiscreteDAE} is sufficiently regular, namely~$u\in H^3(0,T;\R^{\dofOm})$ and~$p\in H^3(0,T;\R^{\dofGa})$. Then, initial data of the form
\begin{equation*}
	\|u^k - u(t^k)\|_{\Mu}^2 + \tau\, \|u^k - u(t^k)\|_{\Ku}^2 + \|p^k - p(t^k)\|_{\Mp}^2 + \tau\, \|p^k - p(t^k)\|_{\Kp}^2 
	\lesssim \tau^5
\end{equation*}  
for $k=0,1,2$ with consistent $u^2, p^2$ implies the error estimate 
\begin{equation*}
	\| u^{n} - u(t^{n})\|_{\Mu}^2 + \| p^{n} - p(t^{n})\|_{\Mp}^2 
	+ \tau \sum_{k=1}^n \Big( \|u^{k} - u(t^{k})\|_{\Ku}^2 + \| p^{k} - p(t^{k})\|_{\Kp}^2 \Big) 
	\lesssim \tau^4
\end{equation*}
for every $n$.
\end{theorem}
\begin{remark}
Although we have used~$\tau \lesssim h$ in the convergence proof, e.g., to guarantee the stability of the scheme in Prop.~\ref{prop:BDF_stability}, such a condition never occurred numerically for the proposed splitting scheme~\eqref{eq:splittingScheme:fullyDiscrete}; see Section~\ref{sect:numerics}. For the auxiliary scheme~\eqref{eqn:BDF_sequential}, however, which is the centerpiece of the convergence analysis, this condition is indeed necessary; see Figure~\ref{fig:exp_BDF_sequentiell} below. 
\end{remark}
\begin{remark}\label{rem:nonlinear}
Theorems~\ref{th:BDF_convergence_order_weak_norm} and~\ref{th:BDF_original} can be extended to systems with locally Lipschitz continuous right-hand sides~$\fu$ and~$\fp$. This means that for every~$(t,u)$ there exist constants $L_{1}, L_{2}, L_{3} \geq 0$ and a radius $r_{\fu}>0$ such that 
\begin{equation*}
		\|\fu(\check t, \check u) - \fu(\tilde t, \tilde u)\|_{\Mu^{-1}}^2 \leq L_{1}\, |\check{t} - \tilde{t}|^2 + 
		L_{2}\, \|\check{u} - \tilde{u}\|_{\Mu}^2 + 
		L_{3}\, \|\check{u} - \tilde{u}\|_{\Ku}^2 
\end{equation*}
for all $(\check t,\check u),(\tilde t,\tilde u)$ in a ball of radius $r_{\fu}$ around $(t,u)$ with norm $(|\cdot|^2+\|\cdot\|^2_{\Mu+\Ku})^{\sfrac 12}$. An analogous condition should be satisfied for the right-hand side~$\fp$, where~$\fp$ is allowed to dependent not only on $t$ and $p$ but also on $u$. The associated convergence proof is given by a combination of the steps of Theorems~\ref{th:BDF_convergence_order_weak_norm}, \ref{th:BDF_original} and \cite[App.~A]{AltKZ22}.
\end{remark}
%
%
\section{Numerical Experiments}\label{sect:numerics}
This section is devoted to the numerical study of the proposed bulk--surface splitting scheme. The computations are performed on an Intel${}^\text{\textregistered}$ Core\texttrademark{} i7-8565U CPU@1.80\,GHz$\times$4 with 48\,GB memory using MATLAB (version R2021b). 

All tests are carried out on the unit disc, i.e., $\Omega = \{x\in \R^2\,|\, x^2_1 + x^2_2 \leq 1\}$, and on the time interval~$[0,1]$, i.e., $T=1$. The spatial meshes are generated by \textsc{DistMesh}; cf.~\cite{PerS04}. The number of degrees of freedom for the bulk problem are given by~$\dofOm = 159, 320, 640, 1290, 2590, 5161$ and correspond to mesh widths~$h_k \approx h_{k-1}/\sqrt{2}$. Moreover, we consider the special case~$\calT_\Omega|_{\Gamma} = \calT_{\Gamma}$. Based on these meshes, the discrete spaces~$\Vh$, $\Qh$, and $\Mh$ equal the appropriate spaces of globally continuous and piecewise linear functions. In particular, $B_p$ equals the mass matrix~$M_\lambda$ of $\calP_1(\calT_\Gamma)$ and $B_u=[0\ M_\lambda]\, P$ holds for some permutation matrix~$P$ such that the last $\dofLam$ entries of its image are associated to the boundary nodes of~$\calT_\Omega$. 
As error measure, we consider the difference between the exact solution (without any discretization) and the fully discrete solution in the discrete versions of the norms of the spaces  
\begin{equation*}
	L^\infty(L^2) \coloneqq L^\infty(0,1;L^2(\Omega) \times L^2(\Gamma)),  
	\qquad
	L^2(H^1) \coloneqq L^2(0,1;H^1(\Omega) \times H^1(\Gamma)).
\end{equation*}

In the first experiment, we consider a linear heat equation with an additional heat equation on the boundary. For the second experiment, we add the nonlinearity~$p^3-p$ to the right-hand side of the boundary equation. Finally, we give an outlook how to construct splitting schemes of higher order. 
%
%
\subsection{Linear heat equation}\label{sect:numerics:linear}
In this first part, we consider the linear model problem introduced in~\cite[Sec.~5.1.2]{AltKZ22}, i.e., the heat equation with dynamic boundary conditions 
	\begin{align*}
		\dot u - \Delta u 
		&= \fu \qquad\text{in } \Omega,\\*
		\dot u - \Delta_\Gamma u + \partial_{\nu} u 
		&= \fp \qquad\text{on } \Gamma.
	\end{align*}	
The right-hand sides~$\fu$ and~$\fp$ are chosen such that $u(t,x,y) = \exp(-t)xy$ equals the exact solution. 

Figure~\ref{fig:exp_lin_BDF_b} illustrates the errors of the proposed splitting scheme~\eqref{eq:splittingScheme:fullyDiscrete} for different spatial mesh widths~$h$ and temporal step sizes~$\tau$. For the initial data at the time points $t=0,\tau,2\tau$, we use the exact solution. 

\begin{figure}
%
%
\begin{tikzpicture}

\begin{axis}[%
width=2.3in,
height=2.0in,
at={(-2.7in,0.in)},
scale only axis,
xmode=log,
xmin=0.000651041666666667,
xmax=0.24,
xminorticks=true,
xlabel style={font=\color{white!15!black}},
xlabel={step size ($\tau$)},
ymode=log,
ymin=1e-04,
ymax=0.07,
yminorticks=true,
ylabel style={font=\color{white!15!black}},
axis background/.style={fill=white},
title style={font=\bfseries},
ylabel={$\|(u,p)-(u^n,p^n)\|_{L^\infty(L^2)}$},
legend columns = 3,
legend style={legend cell align=left, align=left, at={(1.05,1.05)}, anchor=south, draw=white!15!black}
]
\addplot [color=mycolor0, line width=1.0pt, mark=square, mark options={solid, mycolor0}]
  table[row sep=crcr]{%
0.2	0.0133665549296187\\
0.1	0.00622577359408096\\
0.05	0.00650226327754989\\
0.025	0.00671731848441165\\
0.0125	0.00679659412984711\\
0.00625	0.00682706202182971\\
0.003125	0.00683996838404564\\
0.0015625	0.00684586250549157\\
0.00078125	0.00684867710522834\\
};
\addlegendentry{$h=0.20741$\qquad}

\addplot [color=mycolor1, line width=1.0pt, mark=o, mark options={solid, mycolor1}]
  table[row sep=crcr]{%
0.2	0.0140214370253407\\
0.1	0.0034304576940085\\
0.05	0.00286772139739218\\
0.025	0.00300147305725701\\
0.0125	0.0030556647277605\\
0.00625	0.00307453533266609\\
0.003125	0.00308167210299944\\
0.0015625	0.00308465575059922\\
0.00078125	0.00308600488160325\\
};
\addlegendentry{$h=0.14394$\qquad}

\addplot [color=mycolor2, line width=1.0pt, mark=x, mark options={solid, mycolor2}]
  table[row sep=crcr]{%
0.2	0.0145152753035924\\
0.1	0.00307361279839245\\
0.05	0.00153470304857945\\
0.025	0.00158211823909733\\
0.0125	0.00162212857366344\\
0.00625	0.00163585357951924\\
0.003125	0.00164062275459772\\
0.0015625	0.00164245650119919\\
0.00078125	0.00164323621801424\\
};
\addlegendentry{$h=0.093568$}

\addplot [color=mycolor3, line width=1.0pt, mark size=1.3pt, mark=*, mark options={solid, mycolor3}]
  table[row sep=crcr]{%
0.2	0.0148328484832132\\
0.1	0.00317829140614537\\
0.05	0.000859726596852777\\
0.025	0.000733983891590796\\
0.0125	0.000761647569707706\\
0.00625	0.000772430213254033\\
0.003125	0.000775875940193498\\
0.0015625	0.000777048069286817\\
0.00078125	0.000777492890717739\\
};
\addlegendentry{$h=0.067169$\qquad}

\addplot [color=mycolor4, line width=1.0pt, mark=diamond, mark options={solid, mycolor4}]
  table[row sep=crcr]{%
0.2	0.015025180844746\\
0.1	0.00325225510423836\\
0.05	0.000758671533377537\\
0.025	0.000356502734899507\\
0.0125	0.000365738091923696\\
0.00625	0.000374511131151051\\
0.003125	0.00037731325452384\\
0.0015625	0.000378178820728215\\
0.00078125	0.000378470138418686\\
};
\addlegendentry{$h=0.045276$\qquad}

\addplot [color=mycolor5, line width=1.0pt, mark=+, mark options={solid, mycolor5}]
  table[row sep=crcr]{%
0.2	0.0151426023174014\\
0.1	0.00329729700293321\\
0.05	0.000756359966125836\\
0.025	0.000210683175256654\\
0.0125	0.000180205236861992\\
0.00625	0.000186601379004341\\
0.003125	0.000189007779074368\\
0.0015625	0.000189721065900332\\
0.00078125	0.000189938756710747\\
};
\addlegendentry{$h=0.032228$}

\addplot [color=black, dashed, line width=1.0pt]
  table[row sep=crcr]{%
0.2	0.006\\
0.0002	0.6e-08\\
};

\end{axis}

\begin{axis}[%
width=2.3in,
height=2.0in,
at={(0.3in,0.in)},
scale only axis,
xmode=log,
xmin=0.000651041666666667,
xmax=0.24,
xminorticks=true,
xlabel style={font=\color{white!15!black}},
xlabel={step size ($\tau$)},
ymode=log,
ymin=1e-04,
ymax=0.07,
yminorticks=true,
axis background/.style={fill=white},
title style={font=\bfseries},
ylabel={$\|(u,p)-(u^n,p^n)\|_{L^2(H^1)}$},
legend style={at={(0.97,0.03)}, anchor=south east, legend cell align=left, align=left, draw=white!15!black}
]
\addplot [color=mycolor0, line width=1.0pt, mark=square, mark options={solid, mycolor0}]
  table[row sep=crcr]{%
0.2	0.0367098587877462\\
0.1	0.016551720356205\\
0.05	0.0145223749000952\\
0.025	0.0144394362418338\\
0.0125	0.0144863399175039\\
0.00625	0.0145222763457254\\
0.003125	0.0145427936232032\\
0.0015625	0.0145536530836375\\
0.00078125	0.0145592254888811\\
};

\addplot [color=mycolor1, line width=1.0pt, mark=o, mark options={solid, mycolor1}]
  table[row sep=crcr]{%
0.2	0.0345556973587271\\
0.1	0.0105285238790149\\
0.05	0.00696892875003416\\
0.025	0.006673033666029\\
0.0125	0.00666632821103395\\
0.00625	0.00667804079858713\\
0.003125	0.00668640665516593\\
0.0015625	0.00669114782955114\\
0.00078125	0.00669365406969546\\
};

\addplot [color=mycolor2, line width=1.0pt, mark=x, mark options={solid, mycolor2}]
  table[row sep=crcr]{%
0.2	0.0341850149086814\\
0.1	0.00881044360674911\\
0.05	0.00406597044855837\\
0.025	0.0035396284027866\\
0.0125	0.00350039152813342\\
0.00625	0.00350160344062834\\
0.003125	0.00350500218922464\\
0.0015625	0.00350725403254061\\
0.00078125	0.00350851016785227\\
};

\addplot [color=mycolor3, line width=1.0pt, mark size=1.3pt, mark=*, mark options={solid, mycolor3}]
  table[row sep=crcr]{%
0.2	0.0341270145708648\\
0.1	0.00819403011547946\\
0.05	0.00263893091456034\\
0.025	0.0017827479474508\\
0.0125	0.00170420340662119\\
0.00625	0.00169837919012964\\
0.003125	0.00169896894570598\\
0.0015625	0.00169982943369996\\
0.00078125	0.00170038321605921\\
};

\addplot [color=mycolor4, line width=1.0pt, mark=diamond, mark options={solid, mycolor4}]
  table[row sep=crcr]{%
0.2	0.0341575971734974\\
0.1	0.00802316757090761\\
0.05	0.00215088695216242\\
0.025	0.000988802006954868\\
0.0125	0.000853205787782483\\
0.00625	0.000841219373235103\\
0.003125	0.000840319033308392\\
0.0015625	0.00084051421056242\\
0.00078125	0.000840734317930687\\
};

\addplot [color=mycolor5, line width=1.0pt, mark=+, mark options={solid, mycolor5}]
  table[row sep=crcr]{%
0.2	0.0341959548991209\\
0.1	0.00797299413147135\\
0.05	0.00200531309431027\\
0.025	0.000667149501226155\\
0.0125	0.000460252939278551\\
0.00625	0.00044067062843342\\
0.003125	0.000438810076054458\\
0.0015625	0.00043869075717662\\
0.00078125	0.000438758142632782\\
};

\addplot [color=black, dashed, line width=1.0pt]
  table[row sep=crcr]{%
0.2	0.015\\
0.0002	0.15e-07\\
};

\end{axis}
\end{tikzpicture}%
	\caption{Convergence history for the bulk--surface splitting scheme~\eqref{eq:splittingScheme:fullyDiscrete} for different mesh widths. The dashed line indicates order~$2$. }
	\label{fig:exp_lin_BDF_b}
\end{figure}
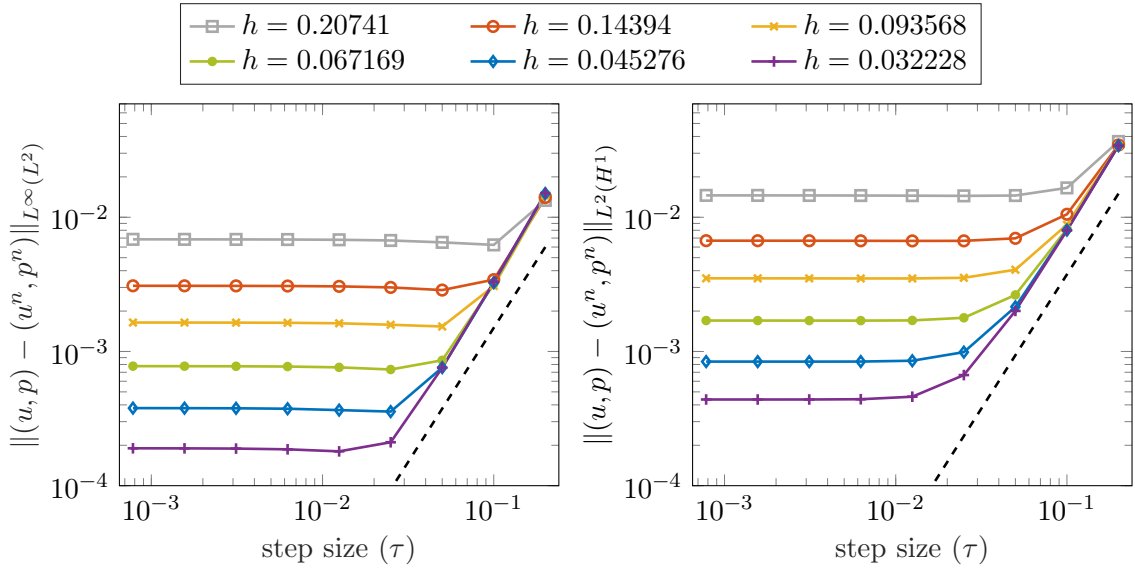
As predicted in Theorem~\ref{th:BDF_original}, we detect second-order convergence for both error measures, i.e., in the~$L^\infty(L^2)$ as well as in the~$L^2(H^1)$-norm. The same holds for the auxiliary method~\eqref{eqn:BDF_sequential}, which was used in the convergence proof; cf.~Theorem~\ref{th:BDF_convergence_order_weak_norm}. The corresponding convergence history is illustrated in Figure~\ref{fig:exp_BDF_sequentiell}. 
\begin{figure}
%
%

\begin{tikzpicture}

\begin{axis}[%
width=2.3in,
height=2.0in,
at={(-2.7in,0.in)},
scale only axis,
xmode=log,
xmin=0.000651041666666667,
xmax=0.24,
xminorticks=true,
xlabel style={font=\color{white!15!black}},
xlabel={step size ($\tau$)},
ymode=log,
ymin=1e-04,
ymax=0.3,
yminorticks=true,
ylabel style={font=\color{white!15!black}},
ylabel={$\|(u,p)-(\hat u^n,p^n)\|_{L^\infty(L^2)}$},
axis background/.style={fill=white},
title style={font=\bfseries},
legend columns = 3,
legend style={legend cell align=left, align=left, at={(1.05,1.05)}, anchor=south, draw=white!15!black}
]
\addplot [color=mycolor0, line width=1.0pt, mark=square]
  table[row sep=crcr]{%
0.2	0.0091214604384479\\
0.1	0.0056170337054697\\
0.05	0.00640050046778132\\
0.025	0.00669369417661752\\
0.0125	0.00679087898553239\\
0.00625	0.00682565446560492\\
0.003125	0.00683961855800284\\
0.0015625	0.00684577512106188\\
0.00078125	0.0068486552891685\\
};
\addlegendentry{$h=0.20741$\qquad}

\addplot [color=mycolor1, line width=1.0pt, mark=o]
  table[row sep=crcr]{%
  	0.2	0.0110958775318057\\
  	0.1	0.00283404444958513\\
  	0.05	0.00278420657417194\\
  	0.025	0.00298465162689718\\
  	0.0125	0.00305161602836487\\
  	0.00625	0.00307354396490142\\
  	0.003125	0.00308142600883779\\
  	0.0015625	0.00308459437717352\\
  	0.00078125	0.00308598952550605\\
};
\addlegendentry{$h=0.14394$\qquad}

\addplot [color=mycolor2, line width=1.0pt, mark=x]
  table[row sep=crcr]{%
0.2	0.0125570098392175\\
0.1	0.0025663222752571\\
0.05	0.00146481557228301\\
0.025	0.00156949501987209\\
0.0125	0.0016192641467432\\
0.00625	0.0016351592477553\\
0.003125	0.00164045080675614\\
0.0015625	0.00164241364956785\\
0.00078125	0.00164322550970279\\
};
\addlegendentry{$h=0.093568$}

\addplot [color=mycolor3, line width=1.0pt, mark size=1.3pt, mark=*]
  table[row sep=crcr]{%
0.2	0.0134520133612516\\
0.1	0.0027857028119642\\
0.05	0.000789975340859919\\
0.025	0.000723766862728458\\
0.0125	0.000759543526969006\\
0.00625	0.000771930938286659\\
0.003125	0.00077575309438802\\
0.0015625	0.000777017505190518\\
0.00078125	0.000777485263617418\\
};
\addlegendentry{$h=0.067169$\qquad}

\addplot [color=mycolor4, line width=1.0pt, mark=diamond]
  table[row sep=crcr]{%
  	0.2	0.0140543441769299\\
  	0.1	0.00297860358860053\\
  	0.05	0.000709217452478693\\
  	0.025	0.00034758752909319\\
  	0.0125	0.000364143492675777\\
  	0.00625	0.000374151696763278\\
  	0.003125	0.000377225950289985\\
  	0.0015625	0.000378157168642964\\
  	0.00078125	0.000378464736016595\\
};
\addlegendentry{$h=0.045276$\qquad}

\addplot [color=mycolor5, line width=1.0pt, mark=+]
  table[row sep=crcr]{%
  	0.2	0.0144570247558384\\
  	0.1	0.00310489543316615\\
  	0.05	0.000725003574098452\\
  	0.025	0.000202272008785534\\
  	0.0125	0.000178946628944052\\
  	0.00625	0.000186340929924153\\
  	0.003125	0.000188945821182341\\
  	0.0015625	0.000189705773345646\\
  	0.00078125	0.00018993494288697\\
};
\addlegendentry{$h=0.032228$}

\addplot [color=black, dashed, line width=1.0pt]
  table[row sep=crcr]{%
0.2	0.004\\
0.0002	0.4e-08\\
};

\end{axis}

\begin{axis}[%
width=2.3in,
height=2.0in,
at={(0.3in,0.in)},
scale only axis,
xmode=log,
xmin=0.000651041666666667,
xmax=0.24,
xminorticks=true,
xlabel style={font=\color{white!15!black}},
xlabel={step size ($\tau$)},
ymode=log,
ymin=1e-04,
ymax=0.3,
yminorticks=true,
axis background/.style={fill=white},
title style={font=\bfseries},
ylabel={$\|(u,p)-(\hat u^n,p^n)\|_{L^2(H^1)}$},
legend style={at={(0.97,0.03)}, anchor=south east, legend cell align=left, align=left, draw=white!15!black}
]
\addplot [color=mycolor0, line width=1.0pt, mark=square, mark options={solid, mycolor0}]
  table[row sep=crcr]{%
  	0.2	0.0577485955036576\\
  	0.1	0.0171979957387987\\
  	0.05	0.0137500982913773\\
  	0.025	0.0141868096106853\\
  	0.0125	0.0144193639708094\\
  	0.00625	0.0145052406157724\\
  	0.003125	0.0145385083327537\\
  	0.0015625	0.0145525788880514\\
  	0.00078125	0.0145589565578211\\
};

\addplot [color=mycolor1, line width=1.0pt, mark=o, mark options={solid, mycolor1}]
  table[row sep=crcr]{%
  	0.2	0.0737881775323538\\
  	0.1	0.0171937644557516\\
  	0.05	0.00685108187756766\\
  	0.025	0.00642171755647673\\
  	0.0125	0.00659040356951464\\
  	0.00625	0.00665823318262033\\
  	0.003125	0.00668139312494036\\
  	0.0015625	0.00668988819152618\\
  	0.00078125	0.00669333822204482\\
};

\addplot [color=mycolor2, line width=1.0pt, mark=x, mark options={solid, mycolor2}]
  table[row sep=crcr]{%
  	0.2	0.0952369819329928\\
  	0.1	0.022178568538052\\
  	0.05	0.00607045569737348\\
  	0.025	0.00356314484261356\\
  	0.0125	0.00346307530969816\\
  	0.00625	0.00348960852583266\\
  	0.003125	0.00350183895701741\\
  	0.0015625	0.00350645304777776\\
  	0.00078125	0.00350830929986398\\
};

\addplot [color=mycolor3, line width=1.0pt, mark size=1.3pt, mark=*, mark options={solid, mycolor3}]
  table[row sep=crcr]{%
  	0.2	0.116603072179002\\
  	0.1	0.0270310790515945\\
  	0.05	0.00662736720166453\\
  	0.025	0.00220660866970798\\
  	0.0125	0.00169337643191617\\
  	0.00625	0.00168725414599897\\
  	0.003125	0.00169567014762945\\
  	0.0015625	0.00169897258896656\\
  	0.00078125	0.00170016696340775\\
};

\addplot [color=mycolor4, line width=1.0pt, mark=diamond, mark options={solid, mycolor4}]
  table[row sep=crcr]{%
  	0.2	0.14239884281975\\
  	0.1	0.0329684113232118\\
  	0.05	0.00797559668318482\\
  	0.025	0.00206297643811864\\
  	0.0125	0.000922894892400339\\
  	0.00625	0.00083406139014911\\
  	0.003125	0.000836938804450571\\
  	0.0015625	0.000839571067504436\\
  	0.00078125	0.000840492402740039\\
};

\addplot [color=mycolor5, line width=1.0pt, mark=+, mark options={solid, mycolor5}]
  table[row sep=crcr]{%
  	0.2	0.172313382496728\\
  	0.1	0.0398486372959009\\
  	0.05	0.00963957749890398\\
  	0.025	0.00238147869669197\\
  	0.0125	0.000699089691896042\\
  	0.00625	0.00044717321888756\\
  	0.003125	0.000435955079599302\\
  	0.0015625	0.000437697887015172\\
  	0.00078125	0.000438492627245151\\
};

\addplot [color=black, dashed, line width=1.0pt]
  table[row sep=crcr]{%
0.2	0.03\\
0.0002	0.3e-07\\
};

\end{axis}
\end{tikzpicture}%
	\caption{Convergence history for the auxiliary method~\eqref{eqn:BDF_sequential} showing up-moving error curves for refinements of the spatial mesh. The dashed line indicates order~$2$.}
	\label{fig:exp_BDF_sequentiell}
\end{figure}
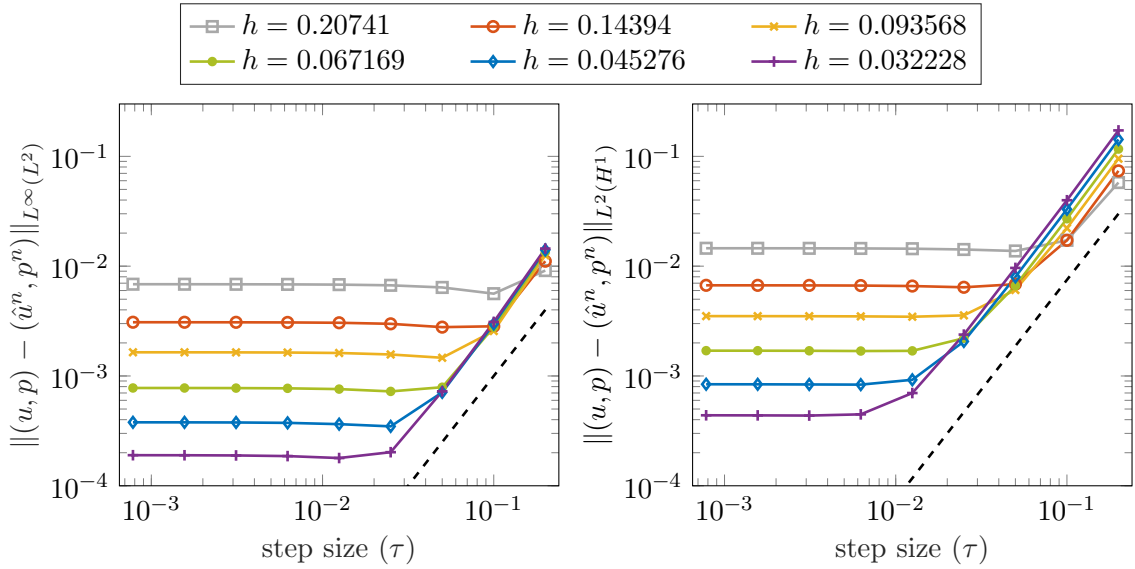
Surprisingly, the error curves of the auxiliary method slightly move upwards with every spatial refinement, which means that convergence in $\tau$ cannot hold independent of the spatial parameter~$h$. In other words, the convergence of method~\eqref{eqn:BDF_sequential} is $h$-dependent and indeed requires a condition of the form~$\tau \lesssim h$. The original scheme, on the other hand, does not show this effect. Hence, the bulk--surface splitting scheme~\eqref{eq:splittingScheme:fullyDiscrete} is independent of $h$. This indicates that the requirement of the weak CFL condition~$\tau \lesssim h$ in Theorem~\ref{th:BDF_original} is a byproduct of our proof technique rather than a necessary requirement. 

In view of an optimal scaling of the second-order splitting scheme in combination with piecewise linear finite elements, a simple calculation shows that the $L^\infty(L^2)$-error is bounded by~$c_1 h^2 + c_2 \tau^2$. Hence, the optimal error rate is obtained for $\tau \propto h$. With the same argumentation, one shows that the relation~$\tau \propto \sqrt{h}$ is optimal for the $L^2(H^1)$-error, since it is bounded by $c_1 h + c_2 \tau^2$. The corresponding numerical experiment is shown in Figure~\ref{fig:optimalScaling}. Within the plot, we consider different mesh sizes $h$ and set the time step size such that~$\tau \propto h$ and~$\tau \propto \sqrt{h}$, respectively. 	
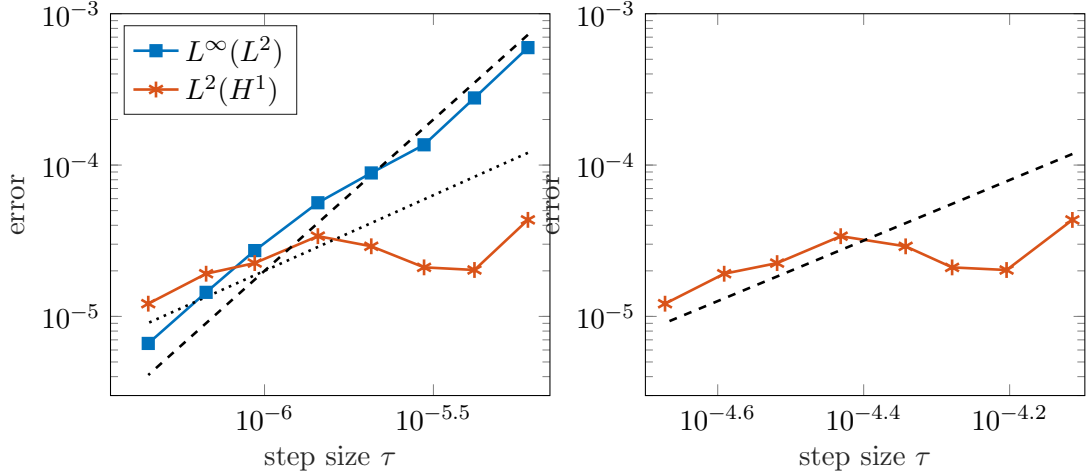
\begin{figure}
	\centering
%
%
\definecolor{mycolor1}{rgb}{0.00000,0.44700,0.74100}%
\definecolor{mycolor2}{rgb}{0.85000,0.32500,0.09800}%
\begin{tikzpicture}

\begin{axis}[%
width=2.3in,
height=2.0in,
at={(-2.7in,0.in)},
scale only axis,
xmode=log,
xmin=3.5e-07,
xmax=7e-06,
xminorticks=true,
xlabel style={font=\color{white!15!black}},
xlabel={step size $\tau$},
ymode=log,
ymin=3e-06,
ymax=0.001,
yminorticks=true,
ylabel style={font=\color{white!15!black}},
ylabel={error},
axis background/.style={fill=white},
legend style={legend cell align=left, align=left, draw=white!15!black, at={(0.03,0.97)}, anchor=north west}
]
\addplot [color=mycolor1, mark=square*, mark options={solid, mycolor1}, line width=1.0pt]
  table[row sep=crcr]{%
6.02409638554217e-06	0.000597043199108543\\
4.18410041841004e-06	0.000277711071883815\\
2.9673590504451e-06	0.000136136050672736\\
2.0703933747412e-06	8.87023315101404e-05\\
1.43884892086331e-06	5.63550823606018e-05\\
9.35453695042095e-07	2.72220316705427e-05\\
6.71591672263264e-07	1.44239739031202e-05\\
4.52693526482571e-07	6.62920328459876e-06\\
};
\addlegendentry{$L^\infty(L^2)$}

\addplot [color=mycolor2, mark=asterisk, mark options={solid, mycolor2}, mark size =3, line width=1.0pt]
  table[row sep=crcr]{%
6.02409638554217e-06	4.34405504497165e-05\\
4.18410041841004e-06	2.02697278672015e-05\\
2.9673590504451e-06	2.1093189924454e-05\\
2.0703933747412e-06	2.90652576194392e-05\\
1.43884892086331e-06	3.39182015832691e-05\\
9.35453695042095e-07	2.25134489366924e-05\\
6.71591672263264e-07	1.91509426558745e-05\\
4.52693526482571e-07	1.21368922635648e-05\\
};
\addlegendentry{$L^2(H^1)$}

\addplot [color=black, dashed, forget plot, line width=1.0pt]
  table[row sep=crcr]{%
6.02409638554217e-06	0.000725794745246044\\
4.18410041841004e-06	0.000350133926226782\\
2.9673590504451e-06	0.000176104394685169\\
2.0703933747412e-06	8.57305745234452e-05\\
1.43884892086331e-06	4.14057243413902e-05\\
9.35453695042095e-07	1.75014723113582e-05\\
6.71591672263264e-07	9.02070748506734e-06\\
4.52693526482571e-07	4.09862857838453e-06\\
};
\addplot [color=black, dotted, forget plot, line width=1.0pt]
  table[row sep=crcr]{%
6.02409638554217e-06	0.000120481927710843\\
4.18410041841004e-06	8.36820083682008e-05\\
2.9673590504451e-06	5.93471810089021e-05\\
2.0703933747412e-06	4.1407867494824e-05\\
1.43884892086331e-06	2.87769784172662e-05\\
9.35453695042095e-07	1.87090739008419e-05\\
6.71591672263264e-07	1.34318334452653e-05\\
4.52693526482571e-07	9.05387052965143e-06\\
};
\end{axis}

\begin{axis}[%
width=2.3in,
height=2.0in,
at={(0.1in,0.in)},
scale only axis,
xmode=log,
xmin=2e-05,
xmax=8e-05,
xminorticks=true,
xlabel style={font=\color{white!15!black}},
xlabel={step size $\tau$},
ymode=log,
ymin=3e-06,
ymax=0.001,
yminorticks=true,
ylabel style={font=\color{white!15!black}},
ylabel={error},
axis background/.style={fill=white},
legend style={legend cell align=left, align=left, draw=white!15!black}
]

\addplot [color=mycolor2, mark=asterisk, mark options={solid, mycolor2}, mark size =3, line width=1.0pt, forget plot]
  table[row sep=crcr]{%
7.69230769230769e-05	4.33995406190124e-05\\
6.25e-05	2.02579515149432e-05\\
5.26315789473684e-05	2.10792837150055e-05\\
4.54545454545455e-05	2.90417807861628e-05\\
3.7037037037037e-05	3.38930136536186e-05\\
3.03030303030303e-05	2.25016922662458e-05\\
2.56410256410256e-05	1.91445825545715e-05\\
2.12765957446809e-05	1.21337161559962e-05\\
};

\addplot [color=black, dashed, forget plot, line width=1.0pt]
  table[row sep=crcr]{%
7.69230769230769e-05	0.000118343195266272\\
6.25e-05	7.8125e-05\\
5.26315789473684e-05	5.54016620498615e-05\\
4.54545454545455e-05	4.13223140495868e-05\\
3.7037037037037e-05	2.74348422496571e-05\\
3.03030303030303e-05	1.83654729109275e-05\\
2.56410256410256e-05	1.31492439184747e-05\\
2.12765957446809e-05	9.05387052965143e-06\\
};
\end{axis}

\end{tikzpicture}%
	\caption{Convergence test for the bulk--surface splitting scheme~\eqref{eq:splittingScheme:fullyDiscrete} under the conditions $\tau \propto h$ (left) and $\tau \propto \sqrt{h}$~(right). The reference lines indicate order~$1$ (dotted) and order~$2$ (dashed).}
	\label{fig:optimalScaling}
\end{figure}

We have also implemented the alternative splitting schemes, which result from the different approximations of~$w$ mentioned in Section~\ref{sect:splitting:delayTerms}. This means that~$w$ is approximated by~\eqref{eq:sub1:bulk:bc:delayA} and~\eqref{eq:sub1:bulk:bc:delayC}, respectively. As before, both approaches are combined with the BDF-$2$ method for the temporal discretization. 
The errors for both schemes are non-distinguishable with the naked eye from the one in Figure~\ref{fig:exp_lin_BDF_b}. Hence, we omit the corresponding plots here. 
In summary, all three methods mentioned in Section~\ref{sect:splitting:secondOrderScheme} show convergence of order two, independent of~$h$. This is especially remarkable for the scheme based on~\eqref{eq:sub1:bulk:bc:delayA}, since the approximation of $w$ is only of first order. 
\begin{remark}\label{rem:midpoint}
Alternatively to the BDF discretization, one may also apply the midpoint rule. However, since equation~\eqref{eq:sub1:bulk:a} is approximated at time~$t^n+\frac \tau 2$, also~$u_2$ and~$w$ have to be approximated at this intermediate time point. Therefore, we consider $\Mla u_2(\,\cdot\, - \tfrac \tau 2) \approx \Bp (\tfrac 3 2 p_\tau - \tfrac 1 2 p_{2\tau})$. As before, we have some freedom how to treat the constraint~\eqref{eq:sub1:bulk:c}. We made numerical tests with the choices~$\Mla w(\,\cdot\, - \tfrac \tau 2) = \tfrac{1}{\tau} \Bp (p_\tau - p_{2\tau})$ as well as~$\Mla w(\,\cdot\, - \tfrac \tau 2) = \tfrac{1}{\tau} \Bp (2p_\tau - 3p_{2\tau} + p_{3\tau})$. Both schemes show second-order convergence but with a strong $h$-dependency. This means that the error curves move upwards for refinements of the spatial mesh; see Figure~\ref{fig:exp_lin_Trapez_a} for the first choice. 
\end{remark}
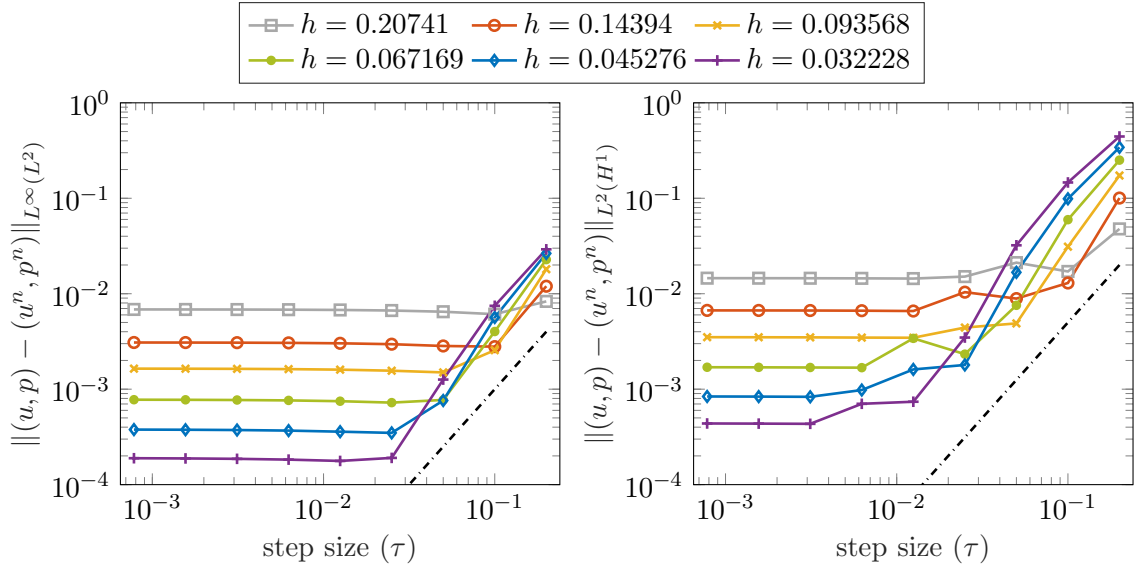
\begin{figure}
%
%
\begin{tikzpicture}

\begin{axis}[%
width=2.3in,
height=2.0in,
at={(-2.7in,0.in)},
scale only axis,
xmode=log,
xmin=0.000651041666666667,
xmax=0.24,
xminorticks=true,
xlabel style={font=\color{white!15!black}},
xlabel={step size ($\tau$)},
ymode=log,
ymin=1e-04,
ymax=1,
yminorticks=true,
ylabel style={font=\color{white!15!black}},
ylabel={$\|(u,p)-(u^n,p^n)\|_{L^\infty(L^2)}$},
axis background/.style={fill=white},
title style={font=\bfseries},
legend columns = 3,
legend style={legend cell align=left, align=left, at={(1.05,1.05)}, anchor=south, draw=white!15!black}
]
\addplot [color=mycolor0, line width=1.0pt, mark=square, mark options={solid, mycolor0}]
  table[row sep=crcr]{%
0.2	0.00831064180357317\\
0.1	0.00610615311425202\\
0.05	0.00647818480407206\\
0.025	0.0066687238716341\\
0.0125	0.00676124573594417\\
0.00625	0.00680655161331665\\
0.003125	0.00682903817104661\\
0.0015625	0.00684024460351847\\
0.00078125	0.00684582404978721\\
};
\addlegendentry{$h=0.20741$}

\addplot [color=mycolor1, line width=1.0pt, mark=o, mark options={solid, mycolor1}]
  table[row sep=crcr]{%
0.2	0.0119853352608639\\
0.1	0.00279125045312459\\
0.05	0.00283545285913072\\
0.025	0.00296049714503073\\
0.0125	0.00302421274935638\\
0.00625	0.00305585915596457\\
0.003125	0.00307157859561799\\
0.0015625	0.00307942647711184\\
0.00078125	0.00308334407234992\\
};
\addlegendentry{$h=0.14394$}

\addplot [color=mycolor2, line width=1.0pt, mark=x, mark options={solid, mycolor2}]
  table[row sep=crcr]{%
0.2	0.0180658057415511\\
0.1	0.00256312479061496\\
0.05	0.00149846315223859\\
0.025	0.00156147663891654\\
0.0125	0.00160167590852741\\
0.00625	0.00162263236455379\\
0.003125	0.00163323615047251\\
0.0015625	0.00163857114022713\\
0.00078125	0.00164124530296537\\
};
\addlegendentry{$h=0.093568$}

\addplot [color=mycolor3, line width=1.0pt, mark size=1.3pt, mark=*, mark options={solid, mycolor3}]
  table[row sep=crcr]{%
0.2	0.022774944834912\\
0.1	0.0040190219193388\\
0.05	0.000771717443423533\\
0.025	0.000723078869652495\\
0.0125	0.000748421449221405\\
0.00625	0.000762825280513239\\
0.003125	0.000770271033039022\\
0.0015625	0.000774041543453997\\
0.00078125	0.000775939246380322\\
};
\addlegendentry{$h=0.067169$}

\addplot [color=mycolor4, line width=1.0pt, mark=diamond, mark options={solid, mycolor4}]
  table[row sep=crcr]{%
0.2	0.0264563736356004\\
0.1	0.00563986883347158\\
0.05	0.000760912570152605\\
0.025	0.000348437130988138\\
0.0125	0.000358695519140015\\
0.00625	0.000368213830523193\\
0.003125	0.000373351582465469\\
0.0015625	0.000375988619866425\\
0.00078125	0.000377322365540206\\
};
\addlegendentry{$h=0.045276$}

\addplot [color=mycolor5, line width=1.0pt, mark=+, mark options={solid, mycolor5}]
  table[row sep=crcr]{%
0.2	0.0291868835877752\\
0.1	0.00746992775329283\\
0.05	0.00126082746443052\\
0.025	0.000190605832995084\\
0.0125	0.000177132256512838\\
0.00625	0.000182874677186677\\
0.003125	0.000186352145237303\\
0.0015625	0.000188181441604145\\
0.00078125	0.000189115117645268\\
};
\addlegendentry{$h=0.032228$}

\addplot [color=black, dashdotted, line width=1.0pt]
  table[row sep=crcr]{%
0.2	0.004\\
0.1	0.001\\
0.05	0.00025\\
0.025	6.25e-05\\
0.0125	1.5625e-05\\
0.00625	3.90625e-06\\
0.003125	9.765625e-07\\
0.0015625	2.44140625e-07\\
0.00078125	6.103515625e-08\\
};

\end{axis}

\begin{axis}[%
width=2.3in,
height=2.0in,
at={(0.3in,0.in)},
scale only axis,
xmode=log,
xmin=0.000651041666666667,
xmax=0.24,
xminorticks=true,
xlabel style={font=\color{white!15!black}},
xlabel={step size ($\tau$)},
ymode=log,
ymin=1e-04,
ymax=1,
yminorticks=true,
ylabel style={font=\color{white!15!black}},
ylabel={$\|(u,p)-(u^n,p^n)\|_{L^2(H^1)}$},
axis background/.style={fill=white},
title style={font=\bfseries},
legend style={at={(0.97,0.03)}, anchor=south east, legend cell align=left, align=left, draw=white!15!black}
]
\addplot [color=mycolor0, line width=1.0pt, mark=square, mark options={solid, mycolor0}]
  table[row sep=crcr]{%
0.2	0.04776346733455\\
0.1	0.0171107518249855\\
0.05	0.0211731205211251\\
0.025	0.015120834757731\\
0.0125	0.0144177779451404\\
0.00625	0.0144920748406581\\
0.003125	0.0145287681765461\\
0.0015625	0.0145469014637626\\
0.00078125	0.0145559107199765\\
};

\addplot [color=mycolor1, line width=1.0pt, mark=o, mark options={solid, mycolor1}]
  table[row sep=crcr]{%
0.2	0.100513128842429\\
0.1	0.0129595755743594\\
0.05	0.00889253286279804\\
0.025	0.0103840428015652\\
0.0125	0.00658780074616006\\
0.00625	0.00664107911338887\\
0.003125	0.00666907015423687\\
0.0015625	0.0066827763540288\\
0.00078125	0.00668954048651512\\
};

\addplot [color=mycolor2, line width=1.0pt, mark=x, mark options={solid, mycolor2}]
  table[row sep=crcr]{%
0.2	0.17317016382902\\
0.1	0.0311032872887439\\
0.05	0.0048876105290276\\
0.025	0.00442184351521405\\
0.0125	0.00345820597101918\\
0.00625	0.00347182091176979\\
0.003125	0.00349047219700399\\
0.0015625	0.00350009441291819\\
0.00078125	0.00350495631993585\\
};

\addplot [color=mycolor3, line width=1.0pt, mark size=1.3pt, mark=*, mark options={solid, mycolor3}]
  table[row sep=crcr]{%
0.2	0.250975830009403\\
0.1	0.0596778151049227\\
0.05	0.00756264571603902\\
0.025	0.00232727371801882\\
0.0125	0.00340766907466935\\
0.00625	0.00168093091400639\\
0.003125	0.00168751730494149\\
0.0015625	0.00169410857978177\\
0.00078125	0.0016975284945202\\
};

\addplot [color=mycolor4, line width=1.0pt, mark=diamond, mark options={solid, mycolor4}]
  table[row sep=crcr]{%
0.2	0.341752878995867\\
0.1	0.0987816149721883\\
0.05	0.016749855050765\\
0.025	0.0017932689343386\\
0.0125	0.00161015289291362\\
0.00625	0.000979666495432022\\
0.003125	0.000831712481922229\\
0.0015625	0.000836162036201092\\
0.00078125	0.000838561588884315\\
};

\addplot [color=mycolor5, line width=1.0pt, mark=+, mark options={solid, mycolor5}]
  table[row sep=crcr]{%
0.2	0.443864003380401\\
0.1	0.146390126619455\\
0.05	0.0321595433460286\\
0.025	0.00347485921647132\\
0.0125	0.000738924857318177\\
0.00625	0.000702777190899137\\
0.003125	0.000433544529964363\\
0.0015625	0.000435724457749669\\
0.00078125	0.000437241044151944\\
};

\addplot [color=black, dashdotted, line width=1.0pt]
  table[row sep=crcr]{%
0.2	0.02\\
0.00078125	3.0517578125e-07\\
};

\end{axis}
\end{tikzpicture}%
	\caption{Convergence history for the bulk--surface splitting scheme based on the midpoint rule described in Remark~\ref{rem:midpoint}, showing again up-moving error curves. The dashed line indicates order~$2$.}
	\label{fig:exp_lin_Trapez_a}
\end{figure}
%
%
\subsection{A semi-linear problem}
\label{sect:numerics:nonlinear}
As second numerical example, we consider a heat equation with nonlinear dynamic boundary conditions including a double-well potential, i.e., 
\begin{align*}
  	\dot u - \Delta u 
  	&= \fu\hphantom{-u^3+u\ } \qquad\text{in } \Omega,\\*
  	\dot u - \Delta_\Gamma u + \partial_{\nu} u 
  	&= \fp-u^3+u \qquad\text{on } \Gamma.
\end{align*}
As exact solution we set $u(t,x,y) = (x^2+y^2)^2\cos(\pi\, t/2)$ and construct the state-independent right-hand sides $\fu$ and $\fp$ accordingly. Moreover, we use a Newton solver for the nonlinear root-finding problem. Figure~\ref{fig:exp_nonlin_BDF_b} validates Remark~\ref{rem:nonlinear} that scheme~\eqref{eq:splittingScheme:fullyDiscrete} is still of second order if the right-hand sides~$\fu$ and $\fp$ are locally Lipschitz continuous in $u$ and $p$, respectively. Furthermore, the errors are again $h$-independent.
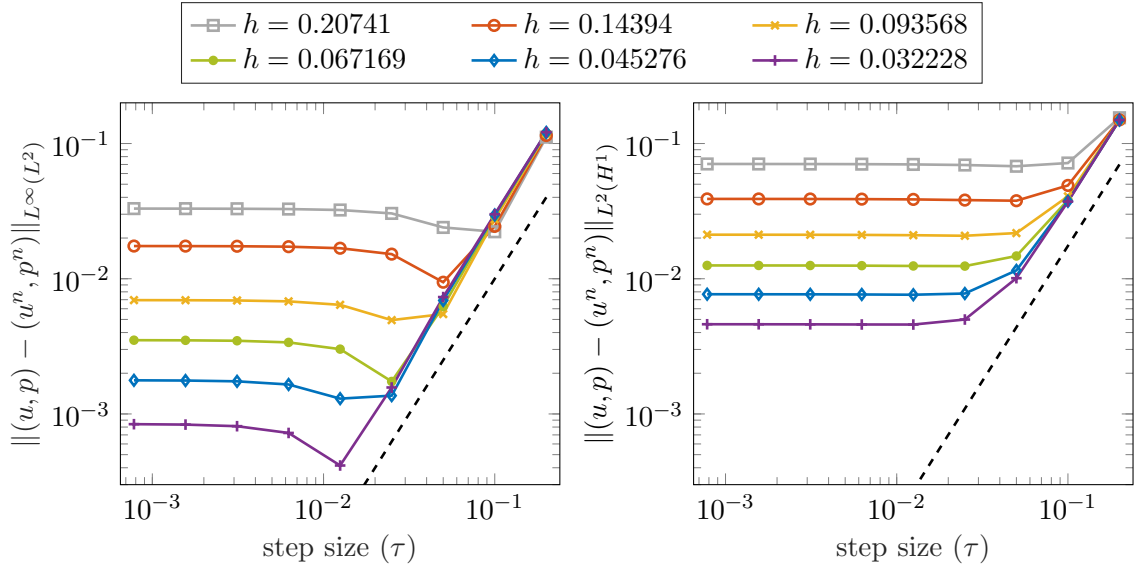
\begin{figure}
%
%
\begin{tikzpicture}

\begin{axis}[%
width=2.3in,
height=2.0in,
at={(-2.7in,0.in)},
scale only axis,
xmode=log,
xmin=0.000651041666666667,
xmax=0.24,
xminorticks=true,
xlabel style={font=\color{white!15!black}},
xlabel={step size ($\tau$)},
ymode=log,
ymin=3e-04,
ymax=0.2,
yminorticks=true,
ylabel style={font=\color{white!15!black}},
axis background/.style={fill=white},
title style={font=\bfseries},
ylabel={$\|(u,p)-(u^n,p^n)\|_{L^\infty(L^2)}$},
legend columns = 3,
legend style={legend cell align=left, align=left, at={(1.05,1.05)}, anchor=south, draw=white!15!black}
]
\addplot [color=mycolor0, line width=1.0pt, mark=square, mark options={solid, mycolor0}]
  table[row sep=crcr]{%
0.2	0.111451351035491\\
0.1	0.0222578837061282\\
0.05	0.0239572758746629\\
0.025	0.0304230176280994\\
0.0125	0.0322372350718074\\
0.00625	0.0327678131398673\\
0.003125	0.0329363638738525\\
0.0015625	0.0329961338938898\\
0.00078125	0.0330198520645871\\
};
\addlegendentry{$h=0.20741$\qquad}

\addplot [color=mycolor1, line width=1.0pt, mark=o, mark options={solid, mycolor1}]
  table[row sep=crcr]{%
0.2	0.114696919442474\\
0.1	0.0244103456157722\\
0.05	0.00942344879262317\\
0.025	0.0152078888585535\\
0.0125	0.0168163363543796\\
0.00625	0.0172617521094217\\
0.003125	0.0173922351388135\\
0.0015625	0.0174341118486044\\
0.00078125	0.0174491720848536\\
};
\addlegendentry{$h=0.14394$\qquad}

\addplot [color=mycolor2, line width=1.0pt, mark=x, mark options={solid, mycolor2}]
  table[row sep=crcr]{%
0.2	0.118202318412039\\
0.1	0.0274841257756267\\
0.05	0.0054780957297579\\
0.025	0.00494219384772666\\
0.0125	0.00641310645448555\\
0.00625	0.006802075736967\\
0.003125	0.00690749981109389\\
0.0015625	0.00693765636628414\\
0.00078125	0.00694706077556096\\
};
\addlegendentry{$h=0.093568$}

\addplot [color=mycolor3, line width=1.0pt, mark size=1.3pt, mark=*, mark options={solid, mycolor3}]
  table[row sep=crcr]{%
0.2	0.119930637166361\\
0.1	0.0288595326358534\\
0.05	0.0062370026478372\\
0.025	0.0017389565777132\\
0.0125	0.00301562349113076\\
0.00625	0.003382170814801\\
0.003125	0.0034783008741718\\
0.0015625	0.00350429294804916\\
0.00078125	0.00351173638715273\\
};
\addlegendentry{$h=0.067169$\qquad}

\addplot [color=mycolor4, line width=1.0pt, mark=diamond, mark options={solid, mycolor4}]
  table[row sep=crcr]{%
0.2	0.120824426885434\\
0.1	0.0296300829392388\\
0.05	0.00690758256136317\\
0.025	0.00136578781893807\\
0.0125	0.00129787812030623\\
0.00625	0.00165175747526998\\
0.003125	0.00174287449068496\\
0.0015625	0.00176668645130244\\
0.00078125	0.00177312158685399\\
};
\addlegendentry{$h=0.045276$\qquad}

\addplot [color=mycolor5, line width=1.0pt, mark=+, mark options={solid, mycolor5}]
  table[row sep=crcr]{%
0.2	0.121315733913015\\
0.1	0.0300124273015311\\
0.05	0.00731962779420763\\
0.025	0.00156926102402107\\
0.0125	0.000415674958341148\\
0.00625	0.000723530541503044\\
0.003125	0.00081174064523872\\
0.0015625	0.000834332367483516\\
0.00078125	0.000840216081347947\\
};
\addlegendentry{$h=0.032228$}

\addplot [color=black, dashed, line width=1.0pt]
  table[row sep=crcr]{%
0.2	0.04\\
0.0002	0.4e-07\\
};

\end{axis}

\begin{axis}[%
width=2.3in,
height=2.0in,
at={(0.3in,0.in)},
scale only axis,
xmode=log,
xmin=0.000651041666666667,
xmax=0.24,
xminorticks=true,
xlabel style={font=\color{white!15!black}},
xlabel={step size ($\tau$)},
ymode=log,
ymin=3e-04,
ymax=0.2,
yminorticks=true,
axis background/.style={fill=white},
title style={font=\bfseries},
ylabel={$\|(u,p)-(u^n,p^n)\|_{L^2(H^1)}$},
legend style={at={(0.97,0.03)}, anchor=south east, legend cell align=left, align=left, draw=white!15!black}
]
\addplot [color=mycolor0, line width=1.0pt, mark=square, mark options={solid, mycolor0}]
  table[row sep=crcr]{%
0.2	0.155061948289216\\
0.1	0.0717517307723613\\
0.05	0.0679138246406753\\
0.025	0.0693586165377729\\
0.0125	0.070076263331403\\
0.00625	0.0703810627096775\\
0.003125	0.0705176518878927\\
0.0015625	0.0705819788904056\\
0.00078125	0.0706129184188987\\
};

\addplot [color=mycolor1, line width=1.0pt, mark=o, mark options={solid, mycolor1}]
  table[row sep=crcr]{%
0.2	0.149656613889319\\
0.1	0.0489915071416882\\
0.05	0.0377593878001279\\
0.025	0.0381732276850644\\
0.0125	0.0386184562172588\\
0.00625	0.0388066126257928\\
0.003125	0.0388867977053364\\
0.0015625	0.0389233589473689\\
0.00078125	0.0389407712761888\\
};

\addplot [color=mycolor2, line width=1.0pt, mark=x, mark options={solid, mycolor2}]
  table[row sep=crcr]{%
0.2	0.149014225765377\\
0.1	0.0406704092658677\\
0.05	0.0217608952988179\\
0.025	0.0207901093607191\\
0.0125	0.0209891154542172\\
0.00625	0.0210967871270442\\
0.003125	0.0211425120588075\\
0.0015625	0.0211629517373943\\
0.00078125	0.0211726587252923\\
};

\addplot [color=mycolor3, line width=1.0pt, mark size=1.3pt, mark=*, mark options={solid, mycolor3}]
  table[row sep=crcr]{%
0.2	0.149006989096373\\
0.1	0.038194027716668\\
0.05	0.0147277511646649\\
0.025	0.0124017134224427\\
0.0125	0.0124295960517149\\
0.00625	0.0124960448463353\\
0.003125	0.0125251784089956\\
0.0015625	0.0125377328669163\\
0.00078125	0.0125435172831162\\
};

\addplot [color=mycolor4, line width=1.0pt, mark=diamond, mark options={solid, mycolor4}]
  table[row sep=crcr]{%
0.2	0.149139576003595\\
0.1	0.0374350038367348\\
0.05	0.0115306142917279\\
0.025	0.00777509721031328\\
0.0125	0.00761545912132173\\
0.00625	0.00765076105946838\\
0.003125	0.00766966284165495\\
0.0015625	0.00767770486645862\\
0.00078125	0.00768129406324557\\
};

\addplot [color=mycolor5, line width=1.0pt, mark=+, mark options={solid, mycolor5}]
  table[row sep=crcr]{%
0.2	0.149255703642409\\
0.1	0.0371820854093411\\
0.05	0.010093746369767\\
0.025	0.00499292378997978\\
0.0125	0.00458377525358579\\
0.00625	0.00458854637764388\\
0.003125	0.00459996354567174\\
0.0015625	0.00460505461354432\\
0.00078125	0.00460727449978018\\
};

\addplot [color=black, dashed, line width=1.0pt]
  table[row sep=crcr]{%
0.2	0.07\\
0.0002	0.7e-07
\\
};

\end{axis}
\end{tikzpicture}%
	\caption{Convergence history for the bulk--surface splitting scheme~\eqref{eq:splittingScheme:fullyDiscrete} applied to the semi-linear example of Section~\ref{sect:numerics:nonlinear}. The dashed line indicates order 2.}
	\label{fig:exp_nonlin_BDF_b}	
\end{figure}

We would like to emphasize that the nonlinearity only appears on the boundary. A positive implication is that the bulk problem, which needs to be solved as part of the splitting scheme, is linear. Without the splitting, one needs to solve a larger nonlinear system in each time step. This difference is crucial, since (for a uniform refinement) $\dofOm$ grows quadratically, whereas~$\dofGa$ only grows linearly. The resulting numerical gain is summarized in Table~\ref{tab:computation_time}. Therein, we compare the computation times of the proposed splitting scheme~\eqref{eq:splittingScheme:fullyDiscrete} and a standard BDF-$2$ discretization of the original system without any splitting, i.e., the approximations of $u_2$ and $w$ are related to $p$ via $\Mla u_2^n = \Bp p^n$ and $\Mla w^n = \Bp \DBDF p^n$. For both schemes, we use a Newton solver with a tolerance of~$10^{-12}$. One can observe that the splitting approach is between three to eight times faster.  
In the simulations, we need for both schemes (in average) two Newton steps for the three smaller step sizes and three steps for the larger ones. Hence, by $\Vh|_\Gamma = \Qh$, the proposed splitting scheme solves in average one linear system of size~$(\dofOm-\dofGa)\times(\dofOm-\dofGa)$ and two (respectively three) systems of size~$\dofGa\times\dofGa$ per time step. In contrast, the fully coupled BDF scheme deals with two (respectively three) systems of size~$\dofOm\times\dofOm$. Further note that $\dofGa$ is neglectable for greater $\dofOm$ and the speed-up is around the number of steps of the nonlinear solver. On the other hand, we have to mention that the BDF method is A-stable \cite[p.~246~f.]{HaiW96}. Therefore, the temporal step size can be choose independently of the spatial mesh width. Moreover, it is a $2$-step rather than a $3$-step scheme.  

\begin{table}
	\caption{Speed-up factors of the splitting scheme~\eqref{eq:splittingScheme:fullyDiscrete} compared to the computation time of the standard $2$-step BDF method applied to the original problem.}
	\label{tab:computation_time}
	\begin{tabular}{c||c|c|c|c|c|c|c|c|c}
		\diagbox{$\ \ \ h$}{$10\,\tau$} & $2^{1}$ & $ 2^{0}$ & $2^{-1}$ & $2^{-2}$ & $2^{-3}$ & $2^{-4}$ & $2^{-5}$ & $2^{-6}$ & $2^{-7}$\\
		\hline
		\hline
		{\color{mycolor0}
		$0.20741\hphantom{0}$} & 3.74 &   3.69 &   3.64 &   3.53 &   3.08 &   3.19 & 3.17 &   3.20 &   3.28 \\
		\hline
		{\color{mycolor1}
		$0.14394\hphantom{0}$} & 4.85 &   4.73 &   4.55 &   4.35 &   3.66 &   3.75 &	3.78 &   3.80 &   3.84 \\
		\hline
		{\color{mycolor2}
		$0.093568$} & 7.34 &   6.10 &   5.77 &   5.72 &   4.74 &   4.90 & 5.14 &   5.25 &   4.94 \\
		\hline
		{\color{mycolor3}
		$0.067169$} & 8.34 &   4.90 &   6.36 &   7.11 &   6.04 &   6.10 & 5.69 &   5.35 &   5.79 \\
		\hline
		{\color{mycolor4}
		$0.045276$} & 5.75 &   6.79 &   5.27 &   5.31 &   3.59 &   3.64 & 4.28 &   3.81 &   3.94 \\
		\hline
		{\color{mycolor5}
		$0.032228$} & 6.79 &   6.36 &   5.81 &   5.81 &   3.91 &   4.27 & 4.03 &   4.14 &   4.21 
	\end{tabular}
\end{table}
%
%
%
%
%
\subsection{A bulk--surface splitting scheme of third order}
\label{sect:numerics:thirdOrder}
Following the ideas presented in Section~\ref{sect:splitting}, we can easily construct bulk--surface splitting schemes of higher order. For this, we decouple bulk and surface dynamics by approximating $p$ and $\dot p$ in the constraints~\eqref{eq:sub1:bulk:b} and~\eqref{eq:sub1:bulk:c} with appropriate order by a linear combination of previous iterates. To obtain a third-order splitting scheme, we can replace the constraints by the approximations 
\begin{align*}
	\Mla u_2(t^{n+4}) = \Bp p(t^{n+4}) &\approx \Bp\big(3p(t^{n+3})-3p(t^{n+2})+p(t^{n+1})\big),\\
	\Mla w(t^{n+4}) = \Bp \dot{p}(t^{n+4}) &\approx \tfrac 1 {6\tau} \Bp \big(26 p(t^{n+3}) - 57 p(t^{n+2}) + 42 p(t^{n+1}) - 11 p(t^{n})\big).
\end{align*}
Note that both approximations are of order three with the least amount of needed data points from the past. For the temporal discretization of the resulting system, we apply the standard BDF-$3$ method, i.e., we use 
\begin{equation*}
	\DBDF x^{n+4} 
	\coloneqq \tfrac1{6\tau} \big(11 x^{n+4} - 18 x^{n+3} + 9 x^{n+2} -2x^{n+1}\big).
\end{equation*}
In the end, we obtain a~$4$-step method, which decouples the dynamics of~$u$ and~$p$.
The application of the resulting scheme to the linear problem introduced in Section~\ref{sect:numerics:linear} indeed shows third-order convergence. Furthermore, the error plots show no dependence on the spatial mesh width~$h$. The corresponding convergence history is illustrated Figure~\ref{fig:exp_lin_BDF3}, where finer meshes with $N_u = 10373, 20232, 41488$ are included in order to depict the actual rate. 
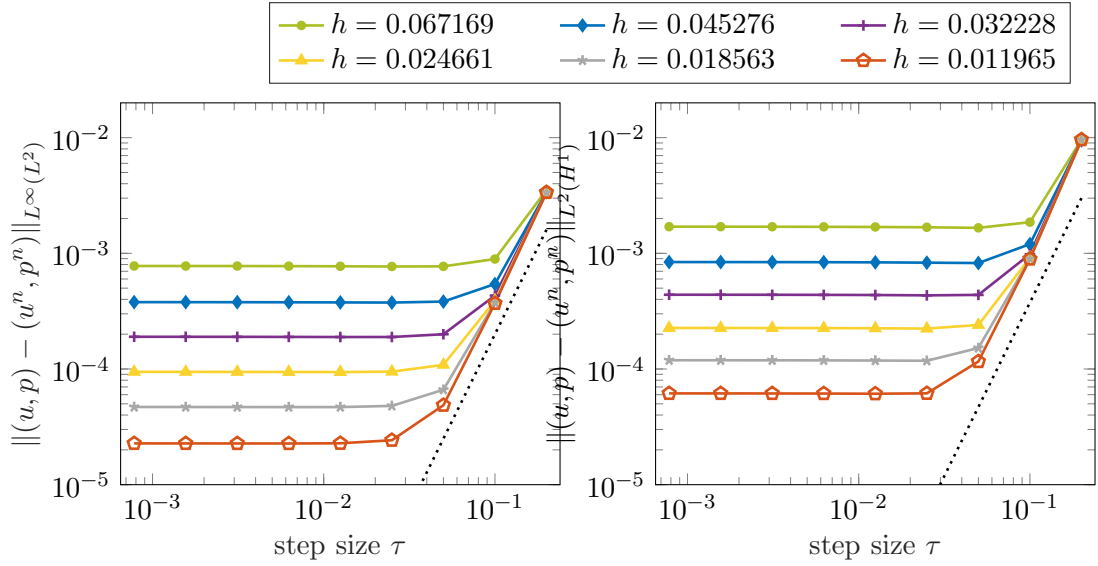
\begin{figure}
%
%

\begin{tikzpicture}

\begin{axis}[%
width=2.3in,
height=2.0in,
at={(-2.7in,0.in)},
scale only axis,
xmode=log,
xmin=0.000651041666666667,
xmax=0.24,
xminorticks=true,
xlabel style={font=\color{white!15!black}},
xlabel={step size $\tau$},
ymode=log,
ymin=1e-05,
ymax=0.02,
yminorticks=true,
ylabel style={font=\color{white!15!black}},
ylabel={$\|(u,p)-(u^n,p^n)\|_{L^\infty(L^2)}$},
axis background/.style={fill=white},
title style={font=\bfseries},
legend columns = 3,
legend style={legend cell align=left, align=left, at={(1.25,1.05)}, anchor=south, draw=white!15!black}
]

%

\addplot [color=mycolor3, line width=1.0pt, mark size=1.3pt, mark=*]
  table[row sep=crcr]{%
0.2	0.00352915502934672\\
0.1	0.000892747339502332\\
0.05	0.00077171558905171\\
0.025	0.000769520849444888\\
0.0125	0.000773142942608169\\
0.00625	0.000775433258958038\\
0.003125	0.000776632090976135\\
0.0015625	0.000777237136433794\\
0.00078125	0.000777540128512165\\
};
\addlegendentry{$h=0.067169$\qquad}

\addplot [color=mycolor4, line width=1.0pt, mark=diamond*]
  table[row sep=crcr]{%
0.2	0.00339527625429293\\
0.1	0.000542855759902989\\
0.05	0.000382788016775397\\
0.025	0.000375377660886339\\
0.0125	0.00037647252313682\\
0.00625	0.000377505231188677\\
0.003125	0.000378077881407069\\
0.0015625	0.000378370676797314\\
0.00078125	0.000378518123942123\\
};
\addlegendentry{$h=0.045276$\qquad}

\addplot [color=mycolor5, line width=1.0pt, mark=+]
  table[row sep=crcr]{%
0.2	0.00335618547213292\\
0.1	0.000431665292081409\\
0.05	0.000199803973470892\\
0.025	0.000189140570920652\\
0.0125	0.000189042997835004\\
0.00625	0.00018948688517869\\
0.003125	0.00018976621998437\\
0.0015625	0.000189912722985268\\
0.00078125	0.000189986765112641\\
};
\addlegendentry{$h=0.032228$}

\addplot [color=mycolor7, line width=1.0pt, mark=triangle*, mark options={solid, mycolor7}]
table[row sep=crcr]{%
	0.2	0.00334956772973942\\
	0.1	0.000390508430459008\\
	0.05	0.000108695308869519\\
	0.025	9.48622568731557e-05\\
	0.0125	9.41208506655805e-05\\
	0.00625	9.42614777872794e-05\\
	0.003125	9.43901494890374e-05\\
	0.0015625	9.44615325172689e-05\\
	0.00078125	9.44980744866564e-05\\
};
\addlegendentry{$h=0.024661$\qquad}

\addplot [color=mycolor0, line width=1.0pt, mark=star]
table[row sep=crcr]{%
0.2	0.00335282635499168\\
0.1	0.000375527574690516\\
0.05	6.61678437300609e-05\\
0.025	4.79867853799769e-05\\
0.0125	4.68593816060595e-05\\
0.00625	4.68475227582196e-05\\
0.003125	4.69015411945722e-05\\
0.0015625	4.69357973230362e-05\\
0.00078125	4.69538104193495e-05\\
};
\addlegendentry{$h=0.018563$\qquad}

\addplot [color=mycolor1, line width=1.0pt, mark=pentagon, mark size=2.5]
table[row sep=crcr]{%
0.2	0.00336127899063175\\
0.1	0.000369643445230191\\
0.05	4.86218001440385e-05\\
0.025	2.42161024171597e-05\\
0.0125	2.27893577229368e-05\\
0.00625	2.26982197418765e-05\\
0.003125	2.27141270938754e-05\\
0.0015625	2.27294264029916e-05\\
0.00078125	2.27379926389452e-05\\
};
\addlegendentry{$h=0.011965$}

\addplot [color=black, dotted, line width=1.0pt]
  table[row sep=crcr]{%
0.2	0.0016\\
0.00078125	9.536743164062502e-11\\
};

\end{axis}

\begin{axis}[%
width=2.3in,
height=2.0in,
at={(0.1in,0.in)},
scale only axis,
xmode=log,
xmin=0.000651041666666667,
xmax=0.24,
xminorticks=true,
xlabel style={font=\color{white!15!black}},
xlabel={step size $\tau$},
ymode=log,
ymin=1e-05,
ymax=0.02,
yminorticks=true,
axis background/.style={fill=white},
title style={font=\bfseries},
ylabel={$\|(u,p)-(u^n,p^n)\|_{L^2(H^1)}$},
legend style={at={(0.97,0.03)}, anchor=south east, legend cell align=left, align=left, draw=white!15!black}
]

%

\addplot [color=mycolor3, line width=1.0pt, mark size=1.3pt, mark=*]
table[row sep=crcr]{%
	0.2	0.00961174035254454\\
	0.1	0.00185968778719092\\
	0.05	0.00166138552675467\\
	0.025	0.00168003185527248\\
	0.0125	0.00169048096598749\\
	0.00625	0.00169571833893103\\
	0.003125	0.00169835647931243\\
	0.0015625	0.00169968394681982\\
	0.00078125	0.00170035035832364\\
};

\addplot [color=mycolor4, line width=1.0pt, mark=diamond*]
table[row sep=crcr]{%
	0.2	0.00949001960332836\\
	0.1	0.00120092646667088\\
	0.05	0.000824382030162911\\
	0.025	0.000830457878251427\\
	0.0125	0.000835790761840234\\
	0.00625	0.000838408346331633\\
	0.003125	0.000839715117667252\\
	0.0015625	0.000840371089479183\\
	0.00078125	0.000840700448667594\\
};

\addplot [color=mycolor5, line width=1.0pt, mark=+]
table[row sep=crcr]{%
	0.2	0.00949789134239025\\
	0.1	0.000983623379786556\\
	0.05	0.000437225534139252\\
	0.025	0.000433257289913822\\
	0.0125	0.000436130933801887\\
	0.00625	0.000437524747063133\\
	0.003125	0.00043821061561406\\
	0.0015625	0.000438553357076131\\
	0.00078125	0.000438725252110458\\
};

\addplot [color=mycolor7, line width=1.0pt, mark=triangle*, mark options={solid, mycolor7}]
table[row sep=crcr]{%
	0.2	0.00953736045626931\\
	0.1	0.000915671713545194\\
	0.05	0.000240288191438501\\
	0.025	0.00022358684303233\\
	0.0125	0.000225026583498601\\
	0.00625	0.000225768921790181\\
	0.003125	0.000226126217257192\\
	0.0015625	0.000226303388641993\\
	0.00078125	0.000226391995609228\\
};

\addplot [color=mycolor0, line width=1.0pt, mark=star]
table[row sep=crcr]{%
0.2	0.00957526886939884\\
0.1	0.000897618831659169\\
0.05	0.000151805328984745\\
0.025	0.000117838768091292\\
0.0125	0.00011831413272269\\
0.00625	0.00011872014792581\\
0.003125	0.0001189107445903\\
0.0015625	0.000119004220383983\\
0.00078125	0.000119050816351981\\
};

\addplot [color=mycolor1, line width=1.0pt, mark=pentagon, mark size=2.5]
table[row sep=crcr]{%
0.2	0.00961521651784442\\
0.1	0.000893352557737638\\
0.05	0.000116076969850662\\
0.025	6.16075487596707e-05\\
0.0125	6.11351109003544e-05\\
0.00625	6.13528726278729e-05\\
0.003125	6.14539182649356e-05\\
0.0015625	6.15025710260367e-05\\
0.00078125	6.15266825995244e-05\\
};

\addplot [color=black, dotted, line width=1.0pt]
table[row sep=crcr]{%
	0.2	0.003\\
	0.002	0.3e-8\\
};

\end{axis}
\end{tikzpicture}%
	\caption{Convergence history for the third-order bulk--surface splitting scheme introduced in Section~\ref{sect:numerics:thirdOrder}. The dotted line indicates order 3.} 
	\label{fig:exp_lin_BDF3}	
\end{figure}
\begin{remark}
Similarly to the discussion in Section~\ref{sect:numerics:linear}, one can compute the optimal scaling of the parameters $\tau$ and $h$. In the here considered third-order case combined with piecewise linear finite elements, this leads to~$\tau \propto h^{2/3}$ for the $L^\infty(L^2)$-error and ~$\tau \propto h^{1/3}$ for the $L^2(H^1)$-error. 	
\end{remark}	
%
%
\section{Conclusion}\label{sect:conclusion}
In this paper, we have introduced a bulk--surface splitting scheme of second order for semi-linear parabolic partial differential equations with dynamic boundary conditions. For this, we have formulated the system as PDAE and approximated the boundary values in the appearing constraints by a linear combination of previous boundary values. To obtain a fully discrete scheme, the resulting system is discretized with a BDF method of corresponding order. 

For the second-order case, we have proven convergence of the resulting $3$-step scheme, assuming a weak CFL-type condition and an appropriate initialization of the method by some other (high-order) time integration method. The proposed construction can also be used to obtain higher-order splitting schemes as illustrated numerically.  
%
%
\bibliographystyle{alpha} 
\bibliography{bib_dynBC}

\begin{thebibliography}{AMU22}

\bibitem[AKZ23]{AltKZ22}
R.~Altmann, B.~Kov\'{a}cs, and C.~Zimmer.
\newblock Bulk--surface {L}ie splitting for parabolic problems with dynamic
  boundary conditions.
\newblock {\em IMA J. Numer. Anal.}, 2(43):950--975, 2023.

\bibitem[Alt19]{Alt19}
R.~Altmann.
\newblock A {PDAE} formulation of parabolic problems with dynamic boundary
  conditions.
\newblock {\em Appl. Math. Lett.}, 90:202--208, 2019.

\bibitem[AMU21]{AltMU21}
R.~Altmann, R.~Maier, and B.~Unger.
\newblock Semi-explicit discretization schemes for weakly-coupled
  elliptic-parabolic problems.
\newblock {\em Math. Comp.}, 90:1089--1118, 2021.

\bibitem[AMU22]{AltMU22ppt}
R.~Altmann, R.~Maier, and B.~Unger.
\newblock Semi-explicit integration of second order for weakly coupled
  poroelasticity.
\newblock {\em ArXiv Preprint}, 2203.16664, 2022.

\bibitem[AV21]{AltV21}
R.~Altmann and B.~Verf\"u{}rth.
\newblock A multiscale method for heterogeneous bulk--surface coupling.
\newblock {\em Multiscale Model. Simul.}, 19(1):374--400, 2021.

\bibitem[AZ23]{AltZ23ppt}
R.~Altmann and C.~Zimmer.
\newblock A posteriori error estimation for parabolic problems with dynamic
  boundary conditions.
\newblock {\em ArXiv Preprint}, 2302.02893, 2023.

\bibitem[BF91]{BreF91}
F.~Brezzi and M.~Fortin.
\newblock {\em Mixed and Hybrid Finite Element Methods}.
\newblock Springer, New York, NY, 1991.

\bibitem[Bra07]{Bra07}
D.~Braess.
\newblock {\em Finite Elements - Theory, Fast Solvers, and Applications in
  Solid Mechanics}.
\newblock Cambridge University Press, New York, NY, third edition, 2007.

\bibitem[CFK23]{CsoFK23}
P.~Csomós, B.~Farkas, and B.~Kovács.
\newblock Error estimates for a splitting integrator for abstract semilinear
  boundary coupled systems.
\newblock {\em IMA J. Numer. Anal.}, published online, 2023.

\bibitem[EF05]{EngF05}
K.-J. Engel and G.~Fragnelli.
\newblock Analyticity of semigroups generated by operators with generalized
  {W}entzell boundary conditions.
\newblock {\em Adv. Differential Equ.}, 10(11):1301--1320, 2005.

\bibitem[Emm04]{Emm04}
E.~Emmrich.
\newblock Stability and convergence of the two-step {BDF} for the
  incompressible {N}avier-{S}tokes problem.
\newblock {\em Int. J. Nonlin. Sci. Num.}, 5(3):199--210, 2004.

\bibitem[Esc93]{Esc93}
J.~Escher.
\newblock Quasilinear parabolic systems with dynamical boundary conditions.
\newblock {\em Commun. Part. Diff. Eq.}, 18(7-8):1309--1364, 1993.

\bibitem[Fai79]{Fai79}
G.~Fairweather.
\newblock On the approximate solution of a diffusion problem by {G}alerkin
  methods.
\newblock {\em J. Inst. Math. Appl.}, 24(2):121--137, 1979.

\bibitem[Gol06]{Gol06}
G.~R. Goldstein.
\newblock Derivation and physical interpretation of general boundary
  conditions.
\newblock {\em Adv. Differential Equ.}, 11(4):457--480, 2006.

\bibitem[GT01]{GilT01}
D.~Gilbarg and N.~S. Trudinger.
\newblock {\em Elliptic Partial Differential Equations of Second Order}.
\newblock Springer-Verlag, Berlin, 2001.

\bibitem[HW96]{HaiW96}
E.~Hairer and G.~Wanner.
\newblock {\em Solving Ordinary Differential Equations {II}: Stiff and
  Differential-Algebraic Problems}.
\newblock Springer-Verlag, Berlin, second edition, 1996.

\bibitem[KA03]{KnaA03}
P.~Knabner and L.~Angermann.
\newblock {\em Numerical Methods for Elliptic and Parabolic Partial
  Differential Equations}.
\newblock Springer, New York, NY, 2003.

\bibitem[KL17]{KovL17}
B.~Kov{\'a}cs and C.~Lubich.
\newblock Numerical analysis of parabolic problems with dynamic boundary
  conditions.
\newblock {\em IMA J. Numer. Anal.}, 37(1):1--39, 2017.

\bibitem[KM06]{KunM06}
P.~Kunkel and V.~Mehrmann.
\newblock {\em Differential-Algebraic Equations. Analysis and Numerical
  Solution}.
\newblock European Mathematical Society Publishing House, Z\"urich, 2006.

\bibitem[Las02]{Las02}
I.~Lasiecka.
\newblock {\em Mathematical Control Theory of Coupled {PDE}s}.
\newblock Society for Industrial and Applied Mathematics (SIAM), Philadelphia,
  PA, 2002.

\bibitem[Lie13]{Lie13}
M.~Liero.
\newblock Passing from bulk to bulk-surface evolution in the {A}llen--{C}ahn
  equation.
\newblock {\em Nonl. Diff. Eqns. Appl. (NoDEA)}, 20(3):919--942, 2013.

\bibitem[MS93]{MatS93}
S.~E. Mattsson and G.~S{\"o}derlind.
\newblock Index reduction in differential-algebraic equations using dummy
  derivatives.
\newblock {\em SIAM J. Sci. Comput.}, 14(3):677--692, 1993.

\bibitem[PS04]{PerS04}
P.-O. Persson and G.~Strang.
\newblock A simple mesh generator in {MATLAB}.
\newblock {\em SIAM Rev.}, 46:329--345, 2004.

\bibitem[Zim21]{Zim21}
C.~Zimmer.
\newblock {\em Temporal {D}iscretization of {C}onstrained {P}artial
  {D}ifferential {E}quations}.
\newblock PhD thesis, Technische Universit{\"a}t Berlin, 2021.

\end{thebibliography}
%
%
\appendix
\section{Estimate of $\tau \sum_{k=3}^{n+3} \|\frac 1 \tau \EBDF^3 e_p^{k}\|_{\Mp}^2$}\label{app:new}
In this appendix, we provide an upper bound of the term  $\tau \sum_{k=3}^{n+3} \|\frac 1 \tau \EBDF^3 e_p^{k}\|_{\Mp}^2$, which completes the proof of Theorem~\ref{th:BDF_convergence_order_weak_norm}. As the first step, we notice that the difference of two consecutive approximations, i.e., $\EBDF \hat u^k$, $\EBDF p^k$, and $\EBDF \lambda^k$, satisfy~\eqref{eqn:BDF_sequential} for $k\geq 4$ with right-hand sides~$\EBDF \fu^k$, $\EBDF \fp^k$. With estimate~\eqref{eqn:BDF_convergence_order_strong_norm_help} we conclude that 
\begin{multline}
	\label{eqn:sum_E3epk_app}
	\tau \sum_{k=4}^{n+3}  \|\tfrac 1 \tau \EBDF^3 e_p^{k}\|^2_{\Mp} 
	\lesssim \|\EBDF e_{\hat u}^3\|_{\Ku}^2 + \|\EBDF e_p^3\|_{\Kp}^2 + \tau\, \|\tfrac 1 \tau \EBDF^2 e_{\hat u}^3\|_{\Mu}^2 + \tau\,\|\tfrac 1 \tau \EBDF^2 e_p^3 \|_{\Mp}^2 + \tfrac{h}{\tau}\, \|\EBDF^3 e_p^3\|_{\Mp}^2\\*
	+ \tau^6 h \int_{t^{0}}^{t^{n+3}} \|p^{(4)}\|_{\Mp}^2 \ds
	+ \tau^5 \int_{t^{1}}^{t^{n+3}} \|u_2^{(3)}\|_{K_{22}}^2 + \tau\, \big( \|u^{(4)}\|_{\Mu}^2 + \|p^{(4)}\|_{\Mp}^2 \big)\ds. 
\end{multline}
To derive the integrals in~\eqref{eqn:sum_E3epk_app}, we have used that 
\begin{equation*}
	\int_\tau^t |r(s)-r(s-\tau)|^2 \ds 
	= \int_\tau^t \big| \int_{s-\tau}^s \dot r(\xi)\dxi \big|^2 \ds 
	\leq \tau \int_{-\tau}^0 \int_{\xi}^{t+\xi} | \dot r(s)|^2 \ds \dxi \leq \tau^2 \int_{0}^t | \dot r(s)|^2 \ds
\end{equation*}
for a sufficiently regular function~$r$. 

As the second step, we bound the terms in the first line of~\eqref{eqn:sum_E3epk_app}. For this, we consider
\begin{align*}
	&\begin{bmatrix}
		\Mu & 0 \\
		0 & \Mp
	\end{bmatrix}
	\begin{bmatrix}
		\DBDF \EBDF e_{\hat u}^{3}\\
		\DBDF \EBDF e_p^{3}
	\end{bmatrix}
	+
	\begin{bmatrix}
		\Ku & 0\\
		0 & \Kp
	\end{bmatrix}
	\begin{bmatrix}
		\EBDF e_{\hat u}^{3}\\
		\EBDF e_p^{3}
	\end{bmatrix}
	+\begin{bmatrix}
		-\Bu^T\\
		\Bp^T
	\end{bmatrix}
	\lambda^{3}\\
	&\qquad = \begin{bmatrix}
		\Mu & 0 \\
		0 & \Mp
	\end{bmatrix}
	\begin{bmatrix}
		\DBDF u(t^{3}) - \dot u(t^3)\\
		\DBDF p(t^{3}) - \dot p(t^3)
	\end{bmatrix}
	+
	\begin{bmatrix}
		\tfrac 3 {2\tau}M_{\bullet 2} \EBDF + K_{\bullet 2}\\
		0
	\end{bmatrix}
	\big( \EBDF^2 e_{\hat u_2}^{3} - \EBDF^2 \hat u_2(t^{3}) \big)\\
	&\qquad\qquad -\begin{bmatrix}
		\Mu & 0 \\
		0 & \Mp
	\end{bmatrix}
	\begin{bmatrix}
		\DBDF e_{\hat u}^{2}\\
		\DBDF  e_p^{2}
	\end{bmatrix}
	-
	\begin{bmatrix}
		\Ku & 0\\
		0 & \Kp
	\end{bmatrix}
	\begin{bmatrix}
		e_{\hat u}^{2}\\
		e_p^{2}
	\end{bmatrix}
\end{align*}
with test function $4\, [(\EBDF e_{\hat{u}}^3)^T\ (\EBDF e_{p}^3)^T]^T$. Combining equation~\eqref{eqn:testing_BDF} and 
\begin{equation*}
	\EBDF \|2x^{n+3}-x^{n+2}\|_M^2 
	= 2\, \|\EBDF x^{n+3}\|_M^2 - 2\, \|\EBDF x^{n+2}\|_M^2 + 2\, \| x^{n+3}\|_M^2 - 3\, \|x^{n+2}\|_M^2 + \| x^{n+1}\|_M^2, 
\end{equation*}
see also \cite[Lem.~4.1]{AltMU22ppt}, we derive the estimate
\begin{align*}
	&\quad\tau\, \big(3\,\|\tfrac{1}{\tau} \EBDF e_{\hat u}^{3}\|^2_{\Mu}
	- 4\,\|\tfrac{1}{\tau} \EBDF e_{\hat u}^{2}\|^2_{\Mu}
	+  \|\tfrac{1}{\tau} \EBDF e_{\hat u}^{1}\|^2_{\Mu} 
	+ 2\EBDF \|\tfrac{1}{\tau} \EBDF^2 e_{\hat u}^{3}\|^2_{\Mu}
	+ \|\tfrac{1}{\tau} \EBDF^3 e_{\hat u}^{3}\|^2_{\Mu}\big)\\*
	&\qquad + \tau\, \big(3\,\|\tfrac{1}{\tau} \EBDF e_{p}^{3}\|^2_{\Mu}
	- 4\, \|\tfrac{1}{\tau} \EBDF e_{p}^{2}\|^2_{\Mu}
	+ \|\tfrac{1}{\tau} \EBDF e_{p}^{1}\|^2_{\Mu} 
	+ 2\EBDF \|\tfrac{1}{\tau} \EBDF^2 e_{p}^{3}\|^2_{\Mu}
	+ \|\tfrac{1}{\tau} \EBDF^3 e_{p}^{3}\|^2_{\Mu} \big)\\
	&\qquad + 4\, \|\EBDF e_{\hat u}^{3}\|_{\Ku}^2 + 4\, \|\EBDF e_p^{3}\|_{\Kp}^2\\*
	&\leq 4\tau\, \|\tfrac 1 \tau \EBDF e_{\hat u}^{3}\|_{\Mu} \Big( \|\DBDF u(t^3)-\dot u (t^3)\|_{\Mu} + \sqrt{c_M h}\, \Big[ \|\tfrac{3}{2\tau} \EBDF^3 e_p^3\|_{\Mp} + \|\tfrac{3}{2\tau} \EBDF^3 p(t^3)\|_{\Mp}\Big]  \\*
	&\qquad + \|\DBDF e_u^2\|_{\Mu} \Big) + 3\, \|\EBDF e_{\hat{u}}^3\|_{\Ku}^2 + 4\, \|\EBDF^2 e_{\hat{u}_2}^3\|_{K_{22}}^2 + 4\, \|\EBDF^2 \hat{u}_2(t^3)\|_{K_{22}}^2 + 4\, \|e_{\hat{u}}^2\|_{\Ku}^2 \\
	&\qquad + 4\tau\, \|\tfrac 1 \tau \EBDF e_{p}^{3}\|_{\Mp} \Big( \|\DBDF p(t^3)-\dot p (t^3)\|_{\Mp} + \|\DBDF e_p^2\|_{\Mp} \Big) + 3\, \|\EBDF e_{p}^3\|_{\Kp}^2 + \tfrac 43\, \|e_{p}^2\|_{\Kp}^2.
\end{align*}
By Young's inequality and Lemma~\ref{lem:error_DandE}, we then get 
\begin{align*}
	&\|\EBDF e_{\hat u}^3\|_{\Ku}^2 + \|\EBDF e_p^3\|_{\Kp}^2 + \tau\, \|\tfrac 1 \tau \EBDF^2 e_{\hat u}^3\|_{\Mu}^2 + \tau\,\|\tfrac 1 \tau \EBDF^2 e_p^3 \|_{\Mp}^2 + (1+h)\tau\, \|\tfrac 1 \tau \EBDF^3 e_p^{3}\|_{\Mp}^2\\*
	&\quad\lesssim (1+h)\, \Big( \|e_{\hat{u}}^2\|_{\Ku}^2 + \|e_{p}^2\|_{\Kp}^2 + \tau\, \|\tfrac 1 \tau \EBDF e_{\hat{u}}^2\|_{\Mu}^2 + \tau\, \|\tfrac 1 \tau \EBDF e_{p}^2\|_{\Mp}^2 \\
	&\qquad\qquad\qquad + \tau\, \|\DBDF e_{\hat{u}}^2\|_{\Mu}^2 + \tau\, \|\DBDF e_{p}^2\|_{\Mp}^2 + h \tau\, \|\tfrac{1}{\tau} \EBDF^2 e_p^2\|_{\Mp}^2 \\
	&\qquad\qquad\qquad + \tau^4\!\! \max_{s\in[t^{1},t^{3}]} \|\ddot{u}_2(s)\|_{K_{22}}^2 + \tau^4 \int_{t^{1}}^{t^3} \|\dddot{u}\|_{\Mu}^2 + \|\dddot{p}\|_{\Mp}^2 \ds
	+h \tau^4 \int_{t^{0}}^{t^3} \|\dddot{p}\|_{\Mp}^2 \ds \Big) 
\end{align*}
by partially shifting the $c_M h$ and $c_K \tau h^{-1}$ terms from the norm of $\tau^{-1}\EBDF^3 e_p^3$ to $\tau^{-1}\EBDF^2 e_p^3$. By the assumption on the initial values in Theorem~\ref{th:BDF_convergence_order_weak_norm}, the right-hand side is bounded by $C(1+h)^2\tau^4$. Together with the first step, this finally shows the desired estimate
\begin{equation*}
	\tau \sum_{k=3}^{n+3} \|\tfrac{1}{\tau}\EBDF^3 p^k\|_{\Mp}^2 
	\lesssim (1+h)^2\,\tau^4.
\end{equation*}
\end{document}